\documentclass[11pt]{amsart}
\usepackage{geometry}                
\usepackage{graphicx}
\usepackage{xcolor}
\usepackage{amssymb,amsmath}
\usepackage{amsthm}
\usepackage{url}
\usepackage{hyperref}
\usepackage{epstopdf}
\usepackage{subcaption}

\newcommand*{\R}{{\mathbb{R}}}

\newtheorem{lemma}{Lemma}
\newtheorem{definition}{Definition}
\newtheorem{remark}{Remark}
\newtheorem{example}{Example}

\title{Affine-invariant ensemble transform methods for logistic regression}
\author{Jakiw Pidstrigach and Sebastian Reich}
\address{Institut f\"ur Mathematik, Universit\"at Potsdam, Karl-Liebknecht-Str. 24/25, 14476 Potsdam}

\begin{document}

\begin{abstract} We investigate the application of ensemble transform approaches to Bayesian inference of logistic regression problems. Our approach relies on appropriate extensions of the popular ensemble Kalman filter and the feedback particle filter to the cross entropy loss function and is based on a well-established homotopy approach to Bayesian inference. The arising finite particle evolution equations as well as their mean-field limits are affine-invariant. Furthermore, the proposed methods can be implemented in a gradient-free manner in case of nonlinear logistic regression and the data can be randomly subsampled similar to mini-batching of stochastic gradient descent. We also propose a closely related SDE-based sampling method which again is affine-invariant and can easily be made gradient-free. Numerical examples demonstrate the appropriateness of the proposed methodologies.
\end{abstract}

\maketitle


%
\section{Introduction}
%

Statistical inference for logistic regression and classification problems is a well studied problem from both parameter optimisation and Bayesian inference perspectives \cite{Bishop06}. While many classical optimisation and Markov chain Monte Carlo (MCMC) methods are applicable, an efficient inference in the context of semi-parametric models still poses computational challenges. With this paper, we address this challenge by investigating the extension of 
coupling-of-measures and ensemble transform methodologies \cite{reichcotter15} from the squared error loss function to the cross entropy loss function typically used for logistic regression and classification.

More specifically, previous work on ensemble methods for Bayesian inference, such as the ensemble Kalman filter (EnKF) \cite{evensen} and the feedback particle filter (FPF) \cite{meyn13,Yang2014,TdWMR17}, have almost exclusively focused on the squared error loss function of the form
\begin{equation}\label{eq:QL}
\Psi_{\rm data}(\theta) = \frac{1}{2} (g(\theta) - t)^{\rm T} \Gamma^{-1} (g(\theta) -t)\,,
\end{equation}
with $g:\R^D \to \R^N$ a forward map, $t\in \R^N$ the data, $\Gamma \in \R^{N\times N}$ the measurement error covariance matrix, 
and $\theta \in \R^D$ the parameters to be estimated. 
Notable exceptions include the application of ensemble Kalman inversion (EKI) \cite{KS19} and the modified EnKF formulation of 
\cite{HLR18} to the training of neural networks with a cross entropy loss function. While these methods seek to minimise a regularised cross entropy loss function, the more recent work \cite{HTL20} has also investigated ensemble-based sampling methods for Bayesian inference in the context of logistic regression and classification. This work relies on the time-stepping of appropriate stochastic differential 
equations (SDEs) and is in line with several other ensemble-based sampling methods such as the ensemble preconditioned MCMC methods of \cite{LMW18}, the ensemble Kalman sampler (EKS) \cite{GIHWS20}, and affine invariant Langevin dynamics (ALDI) \cite{GINR19}. Contrary to such invariance-of-measures based SDE methodologies, the approach taken in this paper is instead 
founded on the homotopy approach to Bayesian inference, as first formulated in a time-continuous framework in 
\cite{daum10,reich10}, and which is close in spirit to the iterative application of the EnKF \cite{SOB12} 
and parameter estimation methods based on sequential Monte Carlo (SMC) methods \cite{BCJ14}.

In addition to expanding the homotopy-based approach to logistic regression, we address two further important concepts, namely affine-invariance of the proposed methods \cite{GW10,LMW18,GINR19} and sub-sampling of data points, as widely used in stochastic gradient descent \cite{LTE19}. Both concepts are used to improve the computational efficiency of optimisation and sampling methods. Take, for example, a  two dimensional Gaussian random variable with mean zero and covariance matrix  
\begin{equation*}
    \Sigma = \left( \begin{array}{cc} \epsilon^2 & 0 \\ 0 & 1 \end{array} \right).
\end{equation*}
A random walk Metropolis--Hastings algorithm will sample inefficiently whenever $\epsilon^2$ is vastly different from one. An affine-invariant modification, on the other hand, will sample this problem as efficiently as if $\epsilon$ were set to $\epsilon = 1$ \cite{GW10}. Sub-sampling replaces the exact gradient of a cost functional by a cheaper to evaluate stochastic approximation, which agrees with the exact gradient in expectation.

We also discuss the possibility for derivative-free implementations \cite{evensen,KS19,GIHWS20,GINR19}, localisation \cite{evensen,reichcotter15} via dropouts \cite{JMLR:v15:srivastava14a}, and efficient linearly implicit time-stepping methods \cite{akir11}.

Our main contributions with regard to Bayesian homotopy methods for logistic regression are:
\begin{itemize}
\item an extension of affine-invariant ensemble transform approaches, such as the EnKF, to logistic regression,
\item an affine-invariant generalisation of the FPF and its application to logistic regression,
\item an extension of data sub-sampling (mini-batches) to Bayesian homotopy methods.
\end{itemize}
In a further step, we combine these homotopy approaches with SDE-based sampling methods in order to derive
\begin{itemize}
\item a derivative-free and affine-invariant SDE-based sampling methods for logistic regression.
\end{itemize}
We demonstrate the appropriateness of the proposed methods by means of a set of numerical experiments. Extensions to nonlinear
and multi-class logistic regression \cite{Bishop06} are also discussed. We also briefly discuss the application to sigmoidal Cox processes \cite{AMMK09}.

The layout of this paper is as follows. The required mathematical background material on both logistic regression is collected in Section \ref{sec:background}. The homotopy approach to Bayesian inference is summarised in Section \ref{sec:homotopy}. There we also present an affine-invariant formulation of the homotopy approach  and discuss  and analyse data sub-sampling in the spirit of \cite{LTE19}. Section \ref{sec:transform} develops three different algorithmic approaches for the implementation of homotopy-based 
Bayesian inference for logistic regression. More specifically, we propose an affine-invariant modification of the FPF and two extensions of the EnKF to logistic regression. We also discuss robust numerical implementations combining dropouts \cite{JMLR:v15:srivastava14a} with localisation \cite{evensen,reichcotter15} and linearly implicit time-stepping methods \cite{akir11}. Section \ref{sec:LD} combines SDE-based sampling methods with a homotopy-based drift term in order to derive a gradient-free and affine-invariant algorithm for Bayesian inference. While the proposed method is only exact for Gaussian measures, it can be used for faster equilibration and approximate inference provided the posterior distribution is close to Gaussian.  Numerical results are presented and evaluated in Section \ref{sec:numerics} and provide a proof-of-concept while a more detailed exploration is left to follow-up work. Several possible extensions of the proposed methods, namely multi-class classification, nonlinear  logistic regression, and sigmoidal Cox processes, are discussed in Section \ref{sec:extensions} followed by a summary statement in Section \ref{sec:conclusions}.

%
\section{Mathematical problem formulation} \label{sec:background}
%

This paper is motivated by the classical logistic regression problem arising from classification into $L>1$ classes, denoted by $C_l$, over an input space $x \in \R^{J}$ \cite{Bishop06}. In order to simplify the exposition, we start with the binary classification problem, that is, $L=2$ and focus on the linear case. The extension to nonlinear and multi-class logistic regression is discussed in Section \ref{sec:extensions}.

The posterior probability for class $C_1$ is defined as a logistic sigmoid 
$$
\sigma (a) = \frac{1}{1+\exp(-a)}
$$
acting on a linear combination of $x$-dependent and vector-valued features  $\phi_x \in \R^D$ 
so that
\begin{equation} \label{eq:conditional}
\pi (C_1|\phi_x) = \sigma \left( \theta^{\rm T} \phi_x \right).
\end{equation}
The model has $D$ adjustable parameters $\theta \in \R^{D}$. The probability of the complementary class $C_2$ is given by $\pi (C_2|\phi_x) = 1-\pi (C_1|\phi_x)$.  Given a data set $\{(x_n,t_n)\}_{n=1}^N$ with labels $t_n \in \{0,1\}$, the negative log-likelihood function is given by the cross-entropy function
\begin{equation} \label{eq:nlf}
\Psi_{\rm data} (\theta) = - \sum_{n=1}^N \left\{ t_n \ln y_n + (1-t_n)\ln (1-y_n)\right\},
\end{equation}
where $y_n := y_{x_n}(\theta)$ with
\begin{equation} \label{eq:y_x}
y_x(\theta) := \sigma \left( \theta^{\rm T} \phi_x \right).
\end{equation}
Here $t_n = 1$ for samples $x_n$ which are assigned to class $C_1$ and $t_n = 0$ for samples from the complementary class $C_2$.

We introduce the further shorthand
$$
\phi_n = \phi_{x_n}
$$
and, using
$$
\frac{\rm d}{{\rm d}a} \sigma(a) = \sigma(a) (1-\sigma(a)),
$$
obtain the gradient of the cross-entropy function
\begin{equation}\label{eq:gradient}
\nabla_\theta \Psi_{\rm data} (\theta) = \sum_{n=1}^N (y_n-t_n) \phi_n \,.
\end{equation}
We rewrite this gradient more compactly as
as
$$
\nabla_\theta \Psi_{\rm data}(\theta) = \Phi \,(y(\theta)-t)
$$
using
$$
y(\theta) = (y_1,\ldots,y_n)^{\rm T} \in \R^N, \quad t = (t_1,\ldots,t_N)^{\rm T} \in \R^N,
\quad \Phi = \left( \phi_1,\ldots,\phi_N \right) \in \R^{D\times N}.
$$ 
The term $y(\theta)-t$ is often called the innovation in the Kalman filter literature. The Hessian matrix of second-order derivative is given by
$$
D^2_\theta \Psi_{\rm data}(\theta) = \sum_{n=1}^N \phi_n y_n (1-y_n) \phi_n^{\rm T} = \Phi R(\theta) \Phi^{\rm T}
$$
where $R(\theta) \in \R^{N\times N}$ denotes the diagonal matrix with entries $r_{nn} = y_n (1-y_n)$. 

Let us denote a minimiser of $\Psi_{\rm data}(\theta)$ by $\theta_{\rm MLE}$. We note that the maximum likelihood estimators (MLEs) are
not uniquely determined in the over-parametrised setting, that is, when $D \gg N$. For that reason and in order to avoid overfitting 
when $D \approx N$, we introduce a Gaussian prior probability density function (PDF) over the parameters via
$$
\pi_{\rm prior}(\theta) \propto \exp(-\Psi_{\rm prior}(\theta)), \qquad \Psi_{\rm prior}(\theta) := \frac{1}{2} (\theta-m_{\rm prior})^{\rm T} 
\Sigma_{\rm prior}^{-1}(\theta-m_{\rm prior}),
$$
that is, the posterior parameter distribution $\pi_{\rm post}$ is given by
\begin{equation} \label{eq:posterior}
\pi_{\rm post}(\theta) \propto \exp(-\Psi_{\rm post}(\theta)), \qquad \Psi_{\rm post}(\theta):= \Psi_{\rm data}(\theta) 
+ \Psi_{\rm prior}(\theta)\,.
\end{equation}
Here the prior mean $m_{\rm prior} \in \R^D$ and covariance matrix $\Sigma_{\rm prior} \in \R^{D\times D}$ need to be 
chosen appropriately. 

It is known that $\Psi_{\rm post}$ is a strongly convex function since $D^2_\theta \Psi_{\rm data}(\theta) \ge 0$ 
and, hence, the MAP estimator, that is,
$$
\theta_{\rm MAP} = \arg \min_{\rm \theta \in \R^D} \Psi_{\rm post}(\theta),
$$
is uniquely defined. Furthermore, assuming the availability of sufficiently many data points, that is $N\gg 1$, the posterior is well approximated by a Gaussian with mean $\theta_{\rm MAP}$ and covariance matrix
$$
\Sigma_{\rm MAP} = \left(D^2_\theta \Psi_{\rm post}(\theta_{\rm MAP})\right)^{-1},
$$
where 
$$
D^2_\theta \Psi_{\rm post}(\theta)  = \Phi R(\theta) \Phi^{\rm T} + \Sigma_{\rm prior}^{-1} 
\in \R^{D \times D}
$$
denotes the Hessian matrix of second-order derivatives. This is an implication of the Bernstein--von Mises theorem \cite{TS15}. 

%
\section{The homotopy approach to Bayesian inference} \label{sec:homotopy}
%

While there are many Monte Carlo methods available for sampling from the posterior $\pi_{\rm post}$, we focus in this paper on
ensemble transform methods based on the homotopy approach \cite{daum10,reich10}
\begin{equation} \label{eq:homotopy_def}
\pi_\tau (\theta) \propto \exp(-\tau \Psi_{\rm data}(\theta)) \,\pi_{\rm prior}(\theta)
\end{equation}
with $\tau \in [0,1]$, that is, $\pi_0 = \pi_{\rm prior}$ and $\pi_{\rm post} = \pi_1$. The PDFs $\pi_\tau$ satisfy the evolution
equation
\begin{equation} \label{eq:BIT}
\partial_\tau \pi_\tau = - (\Psi_{\rm data} - \pi_\tau \left[\Psi_{\rm data}\right])\,\pi_\tau ,
\end{equation}
where $\pi_\tau [f]$ denotes the expectation value of a function $f:\R^D \to \R$ under the PDF $\pi_\tau$.

In order to derive the associated evolution equations for random variables $\theta_\tau \sim \pi_\tau$, we introduce
potentials $V_\tau:\R^D \to \R$, which are defined by the elliptic partial differential equation (PDE)
\begin{equation}\label{eq:PDE}
\nabla_\theta \cdot (\pi_\tau \Sigma_\tau \nabla_\theta V_\tau) = - (\Psi_{\rm data} - \pi_\tau \left[\Psi_{\rm data}\right])\,\pi_\tau,
\end{equation}
where $\Sigma_\tau$ denotes the covariance matrix of $\pi_\tau$, that is,
$$
\Sigma_\tau = \pi_\tau \left[ (\theta-\pi_\tau [\theta])(\theta-\pi_\tau [\theta])^{\rm T} \right].
$$
Hence the evolution equations are given by
\begin{equation} \label{eq:TODE}
\frac{\rm d}{{\rm d}\tau} \theta_\tau = -\Sigma_\tau \nabla_\tau V_\tau (\theta_\tau), \qquad \theta_0 \sim \pi_{\rm prior} ,
\end{equation}
such that
\begin{equation}\label{eq:post}
\theta_1 \sim \pi_{\rm post}.
\end{equation}
See \cite{reich10} for further details.

\begin{remark} If different models $\mathcal{M}_i$, $i=1,\ldots,I_M$, need to be compared via their Bayes factors $\mathbb{P}[\mathcal{M}_i]$ \cite{KR95}, then the homotopy approach (\ref{eq:BIT}) can also be used and gives rise to the evolution equation
$$
\frac{\rm d}{{\rm d}\tau} \mathbb{P}_\tau =
- \pi_\tau\left[\Psi_{\rm data}\right],
$$
$\tau \in [0,1]$, which is to be solved for each model $\mathcal{M}_i$ with $\mathbb{P}_0$ its prior probability, $\pi_\tau$ its transformed parameter PDFs, as defined by (\ref{eq:BIT}), and $\Psi_{\rm data}$ the model's negative log-likelihood function. It then holds that $\mathbb{P}[\mathcal{M}_i] = \mathbb{P}_1$.
\end{remark}

\begin{remark}
It is, of course, feasible to use a stopping time different from $\tau = 1$ in the homotopy approach (\ref{eq:TODE}). This has already been hinted at in in the context of optimisation and Tikhonov regularisations  in \cite{reich10}; while subsequently being fully explored by the EKI methodology. See \cite{KS19,HSS21} for a detailed description in the context of machine learning as well as for an extensive literature survey.
\end{remark}

\subsection{Affine-invariance} \label{sec:affine_invariance}

Note that any vector field $F_\tau(\theta)$, which satisfies $$
\nabla_\theta \cdot (\pi_\tau F_\tau) = - (\Psi_{\rm data} - \pi_\tau \left[\Psi_{\rm data}\right])\,\pi_\tau ,
$$
will lead to an associated evolution equation 
\begin{equation}\label{eq:TODEG}
\frac{\rm d}{{\rm d}\tau} \theta_\tau = -F_\tau (\theta_\tau), \qquad \theta_0 \sim \pi_{\rm prior}
\end{equation}
such that (\ref{eq:post}) holds. The specific choice $F_\tau = \Sigma_\tau \nabla V_\tau$ in (\ref{eq:PDE}) is motivated by the concept of affine-invariance \cite{GW10,GINR19}, which we explain next.

Let
\begin{equation} \label{eq:affine}
\theta= A \overline{\theta} + b
\end{equation}
denote an invertible affine transformation between two variables $\theta\in \R^D$ and $\overline{\theta}\in \R^D$. 
Furthermore, the
Bayesian inference problem for the variable $\theta$ is defined by $\Psi_{\rm data}$ and $\Psi_{\rm prior}$ and the corresponding problem for the transformed variable $\overline{\theta}$ by
$$
\overline{\Psi}_{\rm data}(\overline{\theta}) := \Psi_{\rm data}(A\overline{\theta} + b), 
\qquad \overline{\Psi}_{\rm prior}(\overline{\theta}) := \Psi_{\rm prior}(A\overline{\theta} + b).
$$
Hence the PDFs $\overline{\pi}_\tau$ for $\overline{\theta}$ satisfy
$$
\overline{\pi}_\tau (\overline{\theta}) = \pi_\tau(A\overline{\theta} + b) \,|A|,\qquad \tau \in [0,1]\,.
$$

\begin{definition} \label{def:AI}
A homotopy formulation (\ref{eq:TODEG}) is called affine-invariant if the associated vector fields $F_\tau$ and $\overline{F}_\tau$
satisfy
\begin{equation} \label{eq:affineC}
F_\tau(A\overline{\theta} + b) = A \overline{F}_\tau(\overline{\theta})
\end{equation}
for any suitable choice of $A \in \R^{D\times D}$ and $b \in \R^D$, which implies $\theta_\tau = A \overline{\theta}_\tau 
+ b$  for all $\tau >0$ if it holds at $\tau =0$.
\end{definition}

\begin{lemma}
The choice
\begin{equation} \label{eq:affine_drift}
F_\tau(\theta) = \Sigma_\tau \nabla_\theta V_\tau (\theta)
\end{equation}
leads to an affine-invariant homotopy formulation and the transformed vector field is given by
$$
\overline{F}_\tau(\overline{\theta}) = \overline{\Sigma}_\tau \nabla_{\overline{\theta}} \overline{V}_\tau (\overline{\theta})
$$
with the potential $\overline{V}_\tau$ satisfying the elliptic PDE
\begin{equation} \label{eq:PDE2}
\nabla_{\overline{\theta}} \cdot (\overline{\pi}_\tau \overline{\Sigma}_\tau \nabla_{\overline{\theta}} \overline{V}_\tau) = - (
\overline{\Psi}_{\rm data} - \overline{\pi}_\tau \left[\overline{\Psi}_{\rm data}\right])\,\overline{\pi}_\tau.
\end{equation}
\end{lemma}

\begin{proof} Since the covariance matrices satisfy
$$
\Sigma_\tau = A \overline{\Sigma}_\tau A^{\rm T},
$$
the PDE (\ref{eq:PDE}) transforms into the PDE (\ref{eq:PDE2}) under the affine transformation
(\ref{eq:affine}). Hence we have
$$
\overline{V}_\tau (\overline{\theta}) = V_\tau (A\overline{\theta} + b)
$$
and obtain
\begin{align*}
\Sigma_\tau \nabla_\theta V_\tau (\theta) &= A \overline{\Sigma}_\tau A^{\rm T} \nabla_\theta V_\tau (A\overline{\theta}+b) \\
&= A \overline{\Sigma}_\tau \nabla_{\overline{\theta}} V_\tau (A\overline{\theta}+b) \\
&= A \overline{\Sigma}_\tau \nabla_{\overline{\theta}} \overline{V}_\tau (\overline{\theta}) \,,
\end{align*}
which verifies (\ref{eq:affineC}).
\end{proof}

\noindent
We note that the affine invariance of (\ref{eq:TODE}) is maintained under the forward Euler time discretisation
\begin{equation*} 
\theta_{\tau_{k+1}} = \theta_{\tau_k} - \Delta \tau
\Sigma_{\tau_k} \nabla_\theta V_{\tau_k}(
\theta_{\tau_k})
\end{equation*}
with step-size $\Delta \tau >0$ and $\tau_k = k\Delta \tau$, $k = 0,\ldots,K-1$ such that $K\Delta \tau = 1$.
This immediately follows from property (\ref{eq:affineC}).

\begin{remark}
The choice  of the vector field $F_\tau(\theta)$ in 
(\ref{eq:affine_drift}) is not unique. It is, for example, possible to add any 
drift term $\Sigma_\tau \nabla_\theta U_\tau (\theta)$
with the potential $U_\tau$ chosen such that
\begin{equation} \label{eq:added_drift}
\nabla_\theta \cdot (\pi_\tau \Sigma_\tau \nabla_\theta U_\tau) = 0.
\end{equation}
A possible choice is 
$$
U_\tau (\theta) = \log \pi_\tau + \tau \Psi_{\rm data}(\theta) - \log \pi_{\rm prior}(\theta)
$$
and the associated operator on the left hand side of (\ref{eq:added_drift}) 
corresponds to the nonlinear Fokker--Planck equation arising from the EKS/ALDI mean-field evolution equations \cite{GIHWS20,GINR19} for sampling from the PDF $\pi_\tau$. Because of (\ref{eq:homotopy_def}), $U_\tau (\theta) = \mbox{const.}$ and (\ref{eq:added_drift}) follows trivially. Upon formally replacing the drift term $\Sigma_\tau \nabla_\theta U_\tau (\theta)$ by its EKS/ALDI stochastic mean-field representation, (\ref{eq:TODE}) turns into the affine-invariant stochastic evolution equation
\begin{equation} \label{eq:ALDI}
{\rm d} \theta_\tau = -\Sigma_\tau 
\nabla_\theta \left\{ V_\tau (\theta) + \tau \nabla_\theta \Psi_{\rm data}(\theta)
- \log \pi_{\rm prior}(\theta)\right\}{\rm d}t  + \Sigma^{1/2}_\tau {\rm d}W_\tau,
\end{equation}
where $W_\tau$ denotes standard $D$-dimensional Brownian motion. The associated marginal PDF $\pi_\tau$ still satisfies (\ref{eq:BIT}) provided $\pi_0 = \pi_{\rm prior}$ at initial time $\tau =0$.
\end{remark}
\subsection{Data sub-sampling} \label{sec:random_sub_sample}

If the data sets are large, then it makes sense to randomly sub-sample the data and to only use these subsets in the computation of the gradient (\ref{eq:gradient}). More precisely,
let $\Delta$ denote a randomly selected subset of $\{1,\ldots,N\}$ without replacement of cardinality $N' \le N$. Given such a mini-batch $\Delta$, the full gradient (\ref{eq:gradient}) is replaced by its approximation
$$
\nabla_\theta \Psi_{\rm data} (\theta) \approx \frac{N}{N'}
\sum_{n\in \Delta}(y_n-t_n)\phi_n\,.
$$
More abstractly, data sub-sampling in terms of mini-batches
gives rise to random negative log-likelihood functions $\Psi_{\rm data}^\gamma$, where $\gamma$ characterises the chosen mini-batch
$\Delta^\gamma$ and is the realisation of an appropriate random variable $\Gamma$ such that
\begin{equation} \label{eq:unbiased1}
\mathbb{E}_\Gamma [\Psi_{\rm data}^\gamma (\theta)] = \Psi_{\rm data}(\theta)
\end{equation}
for all parameter values $\theta \in \mathbb{R}^D$. See
\cite{LTE19} for a detailed investigation of data sub-sampling in the context of stochastic gradient descent and related methods. If $\Psi_{\rm data}^\gamma$ is now being used in (\ref{eq:PDE}) instead of $\Psi_{\rm data}$, then we denote the resulting potential by $V_\tau^\gamma$, which is also unbiased, that is,
$$
\mathbb{E}_\Gamma [\nabla_\theta V_\tau^\gamma (\theta)] =
\nabla_\theta V_\tau (\theta)
$$
due to (\ref{eq:unbiased1}). Let us now time-step the implied mean-field ODEs (\ref{eq:TODE}) by the forward Euler method, that is,
\begin{equation} \label{eq:FEM}
\theta_{\tau_{k+1}} = \theta_{\tau_k} - \Delta \tau
\Sigma_{\tau_k} \nabla_\theta V_{\tau_k}^{\gamma_k}(
\theta_{\tau_k})\,,
\end{equation}
where the $\gamma_k$'s are independent realisations of the random variable $\Gamma$. A formal application of the modified equation analysis of \cite{LTE19} leads to the stochastic differential equation (SDE)
\begin{equation}\label{eq:MSDE}
{\rm d}\theta_\tau = -\Sigma_\tau \left\{\nabla_\theta
V_\tau(\theta_\tau)\,{\rm d}\tau - (\Delta \tau \Omega_\tau(\theta_\tau))^{1/2} {\rm d}W_\tau \right\}\,,
\end{equation}
where $W_\tau$ denotes standard $D$-dimensional Brownian motion and
$$
\Omega_\tau (\theta) := 
\mbox{cov}_\Gamma \left[\nabla_\theta V_\tau^{\gamma}(\theta) \right]\,.
$$
Here $\mbox{cov}_\Gamma[f^\gamma]$ denotes the 
covariance matrix of a vector-valued random function $f^\gamma \in \R^D$ with respect to the random variable $\Gamma$. We conclude from (\ref{eq:MSDE}) that the approximation error inflicted on the posterior distribution at $\tau = 1$ can be made arbitrarily small by reducing the time-step $\Delta \tau$ of the forward Euler method (\ref{eq:FEM}). While a rigorous mathematical analysis is left to future work, we discuss an application to Cox point processes in Section \ref{sec:SCP} below and analyse a Gaussian approximation in the Appendix in more detail.

%
\section{Ensemble transform algorithms} \label{sec:transform}
%

We now discuss several Monte Carlo implementations of the proposed affine-invariant homotopy approach to Bayesian inference for
logistic regression. All methods start with a set of $M$ samples $\theta_0^{(i)}$, $i=1,\ldots,M$, from the prior distribution $\pi_0 = \pi_{\rm prior}$.
The three methods considered here differ in the evolution equations for the ensemble $\{\theta_\tau^{(i)}\}$. We introduce
the empirical mean
$$
\hat{m}_\tau = \frac{1}{M} \sum_{i=1}^M \theta_\tau^{(i)}
$$
and the empirical covariance matrix
$$
\hat \Sigma_\tau = \frac{1}{M-1} \sum_{i=1}^M (\theta_\tau^{(i)}-\hat{m}_\tau)(\theta_\tau^{(i)}-\hat{m}_\tau)^{\rm T}
$$
for $\tau \ge 0$. We also introduce the matrix $\Theta_\tau \in \R^{D \times M}$ of ensemble deviations
\begin{equation} \label{eq:deviations}
\Theta_\tau = \left(\theta_\tau^{(1)}-\hat{m}_\tau, \theta_\tau^{(2)}-\hat{m}_\tau, \ldots, \theta_\tau^{(M)}-\hat{m}_\tau\right) ,
\end{equation}
which leads to the compact representation
\begin{equation} \label{eq:covariance}
\hat \Sigma_{\tau} = \frac{1}{M-1} \Theta_\tau \Theta_\tau^{\rm T}
\end{equation}
of the empirical covariance matrix. The empirical measure associated with the ensemble $\{\theta_\tau^{(i)}\}$ at time $\tau$ is
denoted by $\hat{\pi}_\tau$ implying, for example,
$$
\hat{m}_\tau = \hat{\pi}_\tau [\theta].
$$
More generally,
$$
\hat{\pi}_\tau [f] = \frac{1}{M} \sum_{i=1}^M f(\theta_\tau^{(i)})
$$
for an observable $f(\theta)$. We describe an affine-invariant extension of the FPF next. The key building block is a diffusion map approximation to the PDE (\ref{eq:PDE}).

\subsection{Diffusion maps and the FPF}

The right-hand side of the elliptic PDE (\ref{eq:PDE}) leads to the elliptic operator
$$
\Delta_\pi \phi := \frac{1}{\pi} \nabla_\theta \cdot (\pi \Sigma \nabla_\theta \phi)
$$
with covariance matrix
$$
\Sigma = \pi \left[ (\theta-\pi[\theta])(\theta -\pi[\theta])^{\rm T} \right]
$$
and (\ref{eq:PDE}) can be rewritten as
$$
-\Delta_{\pi_\tau} V_\tau = \Psi_{\rm data} - \pi_\tau \left[\Psi_{\rm data}\right].
$$

Following the FPF methodology, as carefully described in \cite{TMM19}, we regularise the elliptic PDE problem (\ref{eq:PDE}) 
through the fixed-point equation
$$
\tilde V_\tau = P_\epsilon \tilde V_\tau + \epsilon (\Psi_{\rm data} - \pi_\tau \left[\Psi_{\rm data}\right])
$$
with regularisation parameter $\epsilon >0$ and semigroup $P_\epsilon = e^{\epsilon \Delta_{\pi_\tau}}$. One then replaces 
$V_\tau$ by $\tilde V_\tau$ in the evolution equation (\ref{eq:TODE}). 

We note that $P_\epsilon$ is the semi-group associated with the SDE
\begin{equation} \label{eq:SDE_DM}
{\rm d}S_\epsilon = \Sigma_\tau \nabla_\theta 
\log \pi_\tau (S_\epsilon) \,{\rm d}\epsilon + \sqrt{2} \Sigma_\tau^{1/2} {\rm d}W_\epsilon,
\end{equation}
where $W_\epsilon$, $\epsilon \ge 0$, is standard Brownian motion in $\R^D$.  
The SDE (\ref{eq:SDE_DM}) is affine invariant and is closely related to the EKS/ALDI mean-field equations \cite{GIHWS20,GINR19}.

The diffusion map approximation $T_\epsilon$ of $P_\epsilon$ is defined as follows:
$$
T_\epsilon f(\theta) = \frac{1}{n_\epsilon (\theta)} \int_{R^D} k_\epsilon(\theta,\theta')f(\theta') \pi_\tau (\theta') {\rm d}\theta'\,,
$$
where $n_\epsilon(\theta) := \int k_\epsilon (\theta,\theta')\pi_\tau (\theta'){\rm d}\theta'$ is the normalisation constant,
$$
k_\epsilon (\theta,\theta') := \frac{g_\epsilon (\theta,\theta')}{\sqrt{\int g_\epsilon(\theta,\theta'') \pi_\tau (\theta''){\rm d}
\theta''} \sqrt{\int g_\epsilon(\theta',\theta'') \pi_\tau (\theta''){\rm d}\theta''}}
$$
and 
\begin{equation} \label{eq:kernel}
g_\epsilon(\theta,\theta') := \exp\left(-\frac{1}{4\epsilon} (\theta-\theta')^{\rm T} \Sigma_\tau^{-1} (\theta-\theta') \right)
\end{equation}
is the Gaussian kernel associated with the pure diffusion in (\ref{eq:SDE_DM}).

The next step is to replace $\pi_\tau$ by the ensemble-induced empirical measure $\hat \pi_\tau$, which results
in the empirical approximation
$$
T_\epsilon^{(M)} f(\theta) := \frac{1}{n_\epsilon^{(M)} (\theta)} \sum_{i=1}^M k_\epsilon^{(M)} (\theta,\theta^{(i)}_\tau)f(\theta^{(i)}_\tau ) \,,
$$
where $n_\epsilon^{(M)}(\theta) := \sum_i k_\epsilon^{(M)} (\theta,\theta^{(i)}_\tau)$ is the normalisation constant, and
$$
k_\epsilon^{(M)} (\theta,\theta') := \frac{g^{(M)}_\epsilon (\theta,\theta')}{\sqrt{\sum_i g^{(M)}_\epsilon(\theta,\theta^{(i)}_\tau)} \sqrt{\sum_i g^{(M)}_\epsilon(\theta',\theta^{(i)}_\tau)}}\,.
$$
Here $g^{(M)}_\epsilon$ denotes the Gaussian kernel (\ref{eq:kernel}) with the covariance matrix $\Sigma_\tau$ replaced by its empirical estimate $\hat \Sigma_\tau$. One finally introduces the $M\times M$ Markov matrix ${\rm T}$ with entries
\begin{equation} \label{eq:DM}
{\rm T}_{ij} = \frac{1}{n_\epsilon^{(M)} (\theta_\tau^{(i)})} k_\epsilon^{(M)} (\theta_\tau^{(i)},\theta_\tau^{(j)})
\end{equation}
and the finite-dimensional fixed-point equation
$$
\tilde{\rm V} = {\rm T}\tilde{\rm V} + \epsilon \,\Delta {\rm \Psi}_{\rm data}
$$
with
$$
\Delta {\rm \Psi}_{\rm data} := \left(\Psi_{\rm data}(\theta_\tau^{(1)})- \hat{\pi}_\tau [\Psi_{\rm data}],\ldots,
\Psi_{\rm data}(\theta_\tau^{(M)})- \hat{\pi}_\tau [\Psi_{\rm data}]\right)^{\rm T} .
$$
Given the solution vector $\tilde{\rm V} \in \R^M$ with entries $\tilde{\rm V}^{(j)}$, $j=1,\ldots,M$,
the approximation to the drift term in the ODE (\ref{eq:TODE}) is now provided as follows \cite{TMM19}:
$$
\hat{\Sigma}_\tau \nabla_\theta \tilde{V}_\tau (\theta) := 
\hat{\Sigma}_\tau \nabla_\theta \left\{ \frac{1}{n_\epsilon^{(M)} (\theta)} \sum_{j=1}^M k_\epsilon^{(M)} (\theta,\theta^{(j)}_\tau) \left(
\tilde{\rm V}^{(j)} + \epsilon \,\Delta \Psi_{\rm data}^{(j)} \right)\right\} .
$$

\begin{lemma}
The affine-invariant FPF drift term for the ensemble members
$\theta_\tau^{(i)}$, $i=1,\ldots,M$, is given by
\begin{equation} \label{eq:driftFPF}
\hat{\Sigma}_\tau \nabla_\theta \tilde{V}_\tau (\theta_\tau^{(i)}) = \sum_{j=1}^M s_{ij} \theta_\tau^{(j)}
\end{equation}
with coefficients
$$
s_{ij} = \frac{1}{2\epsilon} {\rm T}_{ij}\left( r_j - \sum_{k=1}^M {\rm T}_{ik} r_k \right), \qquad r_j := \tilde{\rm V}^{(j)} + \epsilon \,\Delta \Psi_{\rm data}^{(j)}\,.
$$
\end{lemma}

\begin{proof}
The calculations leading to (\ref{eq:driftFPF}) are identical to those for the standard FPF
\cite{TMM19} taking note that the covariance matrix $\hat{\Sigma}_\tau$ cancels out.
\end{proof}

\noindent
The resulting interacting particle system can be propagated in time using the forward Euler method with step size
$\Delta \tau = 1/K$, that is,
\begin{equation} \label{eq:Euler}
\theta_{\tau_{k+1}}^{(i)} = \theta_{\tau_k}^{(i)} - \Delta \tau \hat{\Sigma}_{\tau_k} \nabla_\theta \tilde{V}_{\tau_k} (\theta_{\tau_k}^{(i)}) , \qquad i = 1,\ldots,M,
\end{equation}
$k=0,\ldots,K-1$, with the right-hand side defined by (\ref{eq:driftFPF}). Also note that $M\to \infty$ leads formally to the mean-field forward Euler method (\ref{eq:FEM}).

\begin{remark}
We note that the Markov matrix ${\rm T}$, defined by (\ref{eq:DM}), is non-reversible contrary to the underlying stochastic process (\ref{eq:SDE_DM}). This problem can be addressed by a bi-stochastic approximation, which also leads to improved convergence rates. See \cite{WR20} for more details. 
\end{remark}

\noindent
While the FPF leads to a consistent approximation in the limit $M\to \infty$ and $\epsilon \to 0$, its application is limited to low-dimensional problems. We next describe two approximations that extend to high-dimensional inference problems.

\subsection{Second-order methods} \label{sec:second-order}

One can derive evolution equations for the mean $m_\tau$ and the covariance matrix $\Sigma_\tau$ of $\pi_\tau$ from
equation (\ref{eq:BIT}). Furthermore, if one assumes that $\pi_\tau$ is well approximated by a Gaussian PDF
with mean $m_\tau$ and covariance matrix $\Sigma_\tau$, the following evolution equations for the mean and the covariance matrix arise.

\begin{lemma}
If the temporal PDF $\pi_t$ in (\ref{eq:PDE}) is Gaussian with mean $m_\tau$ and covariance matrix $\Sigma_\tau$, then their evolution equations are given by
$$
\frac{\rm d}{{\rm d}\tau} m_\tau = -\Sigma_\tau \,\pi_\tau [ \nabla_\theta \Psi_{\rm data}] =
- \Sigma_\tau \Phi (\pi_\tau [y] - t), 
$$
and
$$
\frac{\rm d}{{\rm d}\tau} \Sigma_\tau = -\Sigma_\tau \, \pi_\tau [ D^2_\theta \Psi_{\rm data}] \Sigma_\tau =
- \Sigma_\tau \Phi \, \pi_\tau [R] \Phi^{\rm T}\, \Sigma_\tau ,
$$
respectively.
\end{lemma}

\begin{proof}
The evolution equations are derived using the following well-known identity for Gaussian PDFs $\pi_\tau$ (see, for example, \cite{Op98b0}):
\begin{equation} \label{eq:Opper1}
 \pi_\tau[\theta\, (f(\theta)-  \pi_\tau[f])] = \Sigma_\tau  \pi_\tau [\nabla_\theta f]
\end{equation}
as well as
\begin{equation} \label{eq:Opper2}
\pi_\tau[(\theta-\pi_\tau[\theta])(\theta-\pi_\tau[\theta])^{\rm T} (f(\theta)-\pi_\tau[f])] = \Sigma_\tau \pi_\tau [D^2_\theta f]
\Sigma_\tau
\end{equation}
for any scalar-valued smooth function $f:\R^D \to \R$. 
\end{proof}

\noindent
In order to derive corresponding equations for the ensemble $\{\theta_\tau^{(i)}\}$, we introduce
$$
\hat{y}_\tau := \hat{\pi}_\tau [y] = \frac{1}{M} \sum_{j=1}^M y_\tau^{(j)}, \qquad y_\tau^{(j)} := y\left(\theta_\tau^{(j)}\right) \in \R^N\,,
$$
as well as 
$$
\hat{R}_\tau := \hat{\pi}_\tau [R] = \frac{1}{M}\sum_{j=1}^M R_\tau^{(j)},
$$
where the $R_\tau^{(j)} \in \R^{N\times N}$ are diagonal matrices with entries $r_{nn}^{(j)} = y_{\tau,n}^{(j)}\left(1-y_{\tau,n}^{(j)}\right)$ for
$n=1,\ldots,N$
with
$$
y_{\tau,n}^{(j)} = \sigma \left((\theta_\tau^{(j)})^{\rm T} \phi_{x_n}\right)
\in \R\,.
$$

We finally obtain the evolution equations
\begin{equation} \label{eq:SO1a}
\frac{\rm d}{{\rm d}\tau} \hat{m}_\tau = - \hat{\Sigma}_\tau \Phi (\hat{y}_\tau - t)
\end{equation}
for the ensemble mean and
\begin{equation} \label{eq:SO1b}
\frac{\rm d}{{\rm d}\tau} \Theta_\tau = -\frac{1}{2}\hat{\Sigma}_\tau \Phi \hat{R}_\tau \Phi^{\rm T} \Theta_\tau
\end{equation}
for the ensemble deviations (\ref{eq:deviations}). Note that these equations can be used to propagate the ensemble $\{\theta_\tau^{(i)}\}$ regardless of
the Gaussian assumption made for their derivation. In the following section, we will introduce further approximations which allow us to introduce an affine-invariant gradient flow structure.

\subsection{Ensemble Kalman--Bucy filter} \label{sec:enkbf}

We now turn to a formulation with gradient flow structure in the spirit of the ensemble Kalman--Bucy filter (EnKBF) for quadratic loss functions (\ref{eq:QL}) \cite{br10a,reich10,cotterreich13}. Provided that
$$
\hat{y}_\tau = \hat{\pi}_\tau [y] \approx y(\hat{m}_\tau) = y(\hat{\pi}_\tau[\theta])
$$
and
$$
y_\tau^{(i)} = y\left(\theta_\tau^{(i)}\right) \approx \hat{y}_\tau + \hat{R}_\tau \Phi^{\rm T} (\theta_\tau^{(i)}-\hat{m}_\tau) \,,
$$
one formally obtains the approximation 
\begin{equation} \label{eq:KBF1a}
\frac{\rm d}{{\rm d}t} \theta_\tau^{(i)} = - \frac{1}{2} \hat{\Sigma}_\tau \Phi \,(y_\tau^{(i)}+ y(\hat{m}_\tau) - 2t)\,.
\end{equation}
to the second-order equations (\ref{eq:SO1a})--(\ref{eq:SO1b}).

\begin{lemma}
The EnKBF equations (\ref{eq:KBF1a}) are of affine-invariant gradient structure
\begin{equation} \label{eq:KBF2a}
\frac{\rm d}{{\rm d}t} \theta_\tau^{(i)} = -\hat{\Sigma}_\tau \nabla_{\theta^{(i)}} \Psi_{\rm KBF} (\{\theta_\tau^{(j)}\})\,,
\end{equation}
with potential
\begin{equation} \label{eq:potential}
\Psi_{\rm KBF}(\{\theta^{(j)}\}) = \frac{1}{2}\sum_{j=1}^M \Psi_{\rm data}(\theta^{(j)}) + \frac{M}{2}
\Psi_{\rm data}(\hat{m}_\tau)\,.
\end{equation}
\end{lemma}

\begin{proof}
The lemma can be verified by direct calculation and taking note of
$$
\nabla_{\theta^{(i)}} \Psi_{\rm data}(\hat m_\tau) = 
\frac{1}{M} \Phi \,(y(\hat m_\tau)-t)
$$
in particular. 
\end{proof}

\noindent
The equation (\ref{eq:KBF2a}) also possesses an affine-invariant gradient-flow structure in the space of probability measures in the mean-field limit $M\to \infty$. See \cite{cotterreich13,GINR19,GIHWS20} 
for details.

\begin{remark} The authors of \cite{KS19} proposed a different modification of the EnKF for classification problems. In our notation, it corresponds to
$$
\frac{\rm d}{{\rm d}t} \theta_\tau^{(i)} = -\sum_{n=1}^N \left\{
\frac{1}{M-1}\sum_{j=1}^M \left( y_{\tau,n}^{(j)}- 
\frac{1}{M}\sum_{l=1}^M y_{\tau,n}^{(l)} \right)
\left(\frac{t_n}{y_{\tau,n}^{(i)}}- \frac{1-t_n}{1- y_{\tau,n}^{(i)}} \right) \theta_\tau^{(j)} \right\} .
$$
The modified EnKF formulation of \cite{HLR18} uses an implementation that is closer to ours but does 
not actually propagate ensembles. Also recall that both methods are to be used for reducing the loss function 
$\Psi_{\rm data}(\theta)$ via minimisation instead of sampling from the posterior PDF (\ref{eq:posterior}), which is the subject of this paper. \end{remark}

\begin{remark}
In addition to the EnKBF formulation considered so far, there exists another variant which is based on stochastic perturbations in the innovation \cite{evensen,reich10}. We briefly explain how to extend this alternative formulation to logistic regression. We recall that the ensemble deviations $\Theta_\tau$ satisfy the evolution equation (\ref{eq:SO1b}). The same propagation of the first and second-order moments is achieved by the stochastic equations
$$
{\rm d}(\theta_\tau - m_\tau) = - \Sigma_\tau \Psi \left\{
\pi_\tau[R] \Psi^{\rm T}(\theta_\tau-m_\tau)\,{\rm d}\tau 
+ \pi_\tau[R]^{1/2}{\rm d}W_\tau \right\}
$$
in the mean-field limit $M\to \infty$, where $W_\tau$ denotes standard $N$-dimensional Brownian motion. This suggests to replace the finite ensemble size formulation (\ref{eq:KBF1a}) by the interacting system of SDEs
$$
{\rm d}\theta_\tau^{(i)}=-\hat \Sigma_\tau \Psi \left(
y_\tau^{(i)}\,{\rm d}\tau + \hat{R}_\tau \,{\rm d}W_\tau^{(i)}
- t \,{\rm d}\tau\right).
$$
Here $W_\tau^{(i)}$, $i=1,\ldots,M$, denote independent $N$-dimensional Brownian motions.
\end{remark}

\begin{remark}
One can replace the potential (\ref{eq:potential}) by
\[
\Psi_{\rm EnKBF}(\{\theta^{(j)}\}) = \frac{2-\alpha}{2}\sum_{j=1}^M \Psi_{\rm data}(\theta^{(j)}) + \frac{\alpha\,M}{2}
\Psi_{\rm data}(\hat{m}_\tau)\,.
\]
with parameter $\alpha \in [1,2)$. This modification entails a form of variance inflation \cite{evensen,reichcotter15} for $\alpha > 1$ since the implied time evolution of the ensemble deviations (\ref{eq:deviations}) becomes less contractive. 
\end{remark}

\noindent
Again, the forward Euler method (\ref{eq:Euler}) can be used to solve (\ref{eq:KBF1a}) or (\ref{eq:KBF2a}) in time. However, the inherent stiffness of the equations of motion can lead to very small time-steps \cite{akir11}. This issue and the rank deficiency of $\hat \Sigma_\tau$ for $M\le D$ are addressed in the following subsection.

\subsection{Dropout and time-stepping} \label{sec:dropout}

We now discuss two numerical issue that are common to all ensemble-base methods considered in this section. 

First, we observe that whenever $M \le D$, then the empirical covariance matrix $\hat \Sigma_\tau$ is rank-deficient and the 
ensemble $\{\theta_\tau^{(i)}\}$ is propagated in the subspace defined by the ensemble at initial time. The concept
of covariance localisation has been pioneered in the geoscience community in order to lift this restriction 
\cite{evensen,reichcotter15}. In the context of logistic regression we suggest to instead utilise the idea of dropout training
as pioneered in the machine learning community \cite{JMLR:v15:srivastava14a}. We employ the concept
of dropout as follows: In each time step, a randomly chosen number of entries in $\theta_\tau^{(i)}$, $i=1,\ldots,M$, is set to zero. The resulting modified matrix of ensemble deviations (\ref{eq:deviations}) is denoted by $\tilde \Theta_\tau$. The empirical covariance matrix is now defined by
\begin{equation} \label{eq:localisation}
\hat \Sigma_\tau = \frac{1}{(1-\mu)(M-1)} \tilde \Theta_\tau\tilde \Theta_\tau^{\rm T}\,,
\end{equation}
where $\mu\in (0,1)$ denotes the fraction of entries randomly set to zero. All other aspects of the previously considered algorithms remain unaltered.  The estimator (\ref{eq:localisation}) underestimates the cross-correlations. 
An unbiased estimator requires an additional rescaling of the off-diagonal entries by $1/(1-\mu)$.

\begin{remark} The random dropout approximation (\ref{eq:localisation}) does not provide an unbiased approximation to the empirical covariance matrix 
(\ref{eq:covariance}). An unbiased approximation would require that the off-diagonal entries of $\tilde \Theta_\tau\tilde \Theta_\tau^{\rm T}$  get scaled by $1/((1-\mu)^2(M-1))$ in (\ref{eq:localisation}). While such an unbiased approximation fits into the modified equation analysis framework of Section \ref{sec:random_sub_sample}, we found that (\ref{eq:localisation}) works well for practical applications as its puts a stronger emphasis on the diagonal entries of the covariance matrix, which is in line with standard covariance localisation techniques \cite{evensen,reichcotter15}.
\end{remark}

\begin{remark} There are alternative methods for breaking the sub-space property of ensemble transform methods in case of $M\le D$. For example,
the authors of \cite{KS19} suggest to randomly perturb the ensemble members after each time step of EKI. In our context, this approach can be viewed as combining the EnKBF with diffusion at low temperature. However, it is not clear that diffusion is sufficient to effectively explore the full parameter space in the context of the homotopy approach considered in this paper. 
\end{remark}

\noindent
Second, we return to the issue of efficient time-stepping of the interacting particle systems. We had previously considered
the forward Euler method (\ref{eq:Euler}). However, the method can encounter severe step-size restrictions due to
its restricted domain of stability. Following \cite{akir11}, we therefore consider the following tamed version of the forward Euler discretisation (\ref{eq:Euler}) for the EnKBF formulation from Section \ref{sec:enkbf}. 
Starting from (\ref{eq:KBF1a}) we employ
\begin{equation}
\theta_{\tau_{k+1}}^{(i)} = \theta_{\tau_{k}}^{(i)} - \frac{\Delta \tau}{2}
\hat{\Sigma}_{\tau_k} \Phi \left( I+\Delta \tau \hat{R}_{\tau_k} \Phi^{\rm T} \hat{\Sigma}_{\tau_k} \Phi \right)^{-1} (y_{\tau_k}^{(i)}+ y(\hat{m}_{\tau_k}) - 2t) \label{eq:tamed_integrator}
\end{equation}
with step size $\Delta \tau>0$ for $k=0,\ldots,K-1$. If the necessary matrix inversions become too computationally expensive, one can diagonalise the inverse as proposed in \cite{akir11}. The same modification can be applied to the second-order formulation (\ref{eq:SO1a})-(\ref{eq:SO1b}).

%
\section{Ensemble transform Langevin dynamics} \label{sec:LD}
%

While we have focused on the homotopy approach in the previous section, we now combine the homotopy approach with overdamped Langevin dynamics in order to sample from the desired target distribution (\ref{eq:posterior}) as $\tau \to \infty$. Let us therefore turn the transport-based differential equation (\ref{eq:TODE}) into the McKean--Vlasov SDE 
\begin{equation} \label{eq:KVSDE}
{\rm d} \theta_\tau = -\Sigma_\tau \left\{ \nabla_\theta V_\tau (\theta_\tau)  + \frac{1}{2}
\Sigma_{\rm prior}^{-1}(\theta_\tau+m_\tau -2m_{\rm prior})\right\}  {\rm d}\tau + \Sigma_\tau^{1/2} {\rm d}W_\tau
\end{equation}
by adding diffusion and a prior-related drift term to (\ref{eq:TODE}). This SDE is motivated by the  affine-invariant EKS \cite{GIHWS20} and ALDI \cite{GINR19} variants of standard Langevin dynamics. Note that the crucial differences are that (i) EKS/ALDI uses the gradient $\nabla_\theta \Psi_{\rm prior}$ of the prior in the drift term while (\ref{eq:KVSDE}) contains
$$
\frac{1}{2} \Sigma_{\rm prior}^{-1}(\theta_\tau+m_\tau -2m_{\rm prior})
$$
instead, and (ii), in a similar vein, the potential $\Psi_{\rm data}(\theta)$ is replaced in (\ref{eq:KVSDE}) by a time-dependent potential $V_\tau$ which is defined as the solution of the elliptic PDE (\ref{eq:PDE}). We will find that these modifications are motivated by the two facts that (i) one can find efficient gradient-free time-stepping methods for the SDE (\ref{eq:KVSDE}) and (ii) that the SDE (\ref{eq:KVSDE}) samples from the posterior PDF in the case of Gaussian distributions as we demonstrate next.

\begin{lemma} If the initial $\pi_0$ is Gaussian and
$\Psi_{\rm data}$ is a potential of the form
$$
\Psi_{\rm data}(\theta) = \frac{1}{2} (G\theta - t)^{\rm T} \Gamma^{-1} (G\theta - t),
$$
then the marginal PDFs $\pi_\tau$ of (\ref{eq:KVSDE}) remain Gaussian for all $\tau >0$, and the SDE (\ref{eq:KVSDE}) has a stationary measure which is given by $\pi_\infty(\theta) \propto \exp(-\Psi_{\rm post}(\theta))$.
\end{lemma}

\begin{proof} Under the stated assumption, we have
$$
-\frac{1}{\pi_\tau(\theta)}\nabla_\theta \cdot (\pi_\tau(\theta) \Sigma_\tau \nabla_\theta V_\tau(\theta)) =
(\theta-m_\tau)^{\rm T} \nabla_\theta V_\tau (\theta) - \nabla_\theta \cdot \Sigma_\tau \nabla_\theta V_\tau(\theta)
$$
as well as
\begin{align*}
\Psi_{\rm data}(\theta) - \pi_\tau [\Psi_{\rm data}] &=
\frac{1}{2} (G\theta-t)^{\rm T} \Gamma^{-1} (G\theta-t) \,\,-\\
& \qquad \qquad \frac{1}{2} \pi_\tau[\theta^{\rm T}G^{\rm T} \Gamma^{-1} G\theta] + m_\tau^{\rm T}G^{\rm T}\Gamma^{-1} t - \frac{1}{2} t^{\rm T}\Gamma^{-1} t\\
&= \frac{1}{2} (\theta-m_\tau)^{\rm T} G^{\rm T} \Gamma^{-1}(G\theta + Gm_\tau - 2t) \,\,-\\
&\qquad \qquad 
 \frac{1}{2} \pi_\tau[\theta^{\rm T}G^{\rm T}\Gamma^{-1} G\theta] + \frac{1}{2} m_\tau^{\rm T} G^{\rm T} \Gamma^{-1} G m_\tau .
\end{align*}
Equating both expressions leads us to
\begin{equation} \label{eq:linear_grad}
\nabla_\theta V_\tau (\theta) = \frac{1}{2} G^{\rm T} \Gamma^{-1} (G\theta + Gm_\tau - 2t).
\end{equation}
The evolution equations arising from (\ref{eq:KVSDE}) for the mean $m_\tau$ and the covariance matrix $\Sigma_\tau$ 
are therefore given by
\begin{align*}
\frac{\rm d}{{\rm d}\tau} m_\tau &= -\Sigma_\tau \left(G^{\rm T} \Gamma^{-1}(Gm_\tau -t ) + \Sigma_{\rm prior}^{-1}(m_\tau - m_{\rm prior})\right)\\
&= - \Sigma_\tau \Sigma_{\rm post}^{-1}(m_\tau - m_{\rm post})
\end{align*}
and
\begin{align*}
\frac{\rm d}{{\rm d}\tau}\Sigma_\tau &= -\Sigma_\tau \left(G^{\rm T} \Gamma^{-1}G + \Sigma_{\rm prior}^{-1}\right) \Sigma_\tau 
+\Sigma_\tau \\ &=  - \Sigma_\tau \Sigma_{\rm post}^{-1} \Sigma_\tau + \Sigma_\tau,
\end{align*}
respectively. The associated equilibrium solutions are 
$$
m_\infty = m_{\rm post} := m_{\rm prior} - \Sigma_{\rm post}G^{\rm T} \Gamma^{-1} (Gm_{\rm prior} - t)
$$
and 
$$
\Sigma_\infty = \Sigma_{\rm post} := \left(G^{\rm T} \Gamma^{-1}G + \Sigma_{\rm prior}^{-1}\right)^{-1},
$$
as desired. We also note that the linearity of the gradient (\ref{eq:linear_grad}) implies that the PDFs $\pi_\tau$
remain Gaussian. Finally, (\ref{eq:linear_grad}) is indeed the gradient of the potential
$$
V_\tau (\theta) = \frac{1}{4} (G\theta - t)^{\rm T} \Gamma^{-1} (G\theta - t)
+ \frac{1}{2} \theta^{\rm T} G^{\rm T} \Gamma^{-1} (G m_\tau - t)\,,
$$
which concludes the proof.
\end{proof}

\begin{remark} \label{rem:KVSDE}
One can replace the McKean--Vlasov SDE (\ref{eq:KVSDE}) by the following 
formulation
\begin{equation}  \label{eq:KVSDE2}
{\rm d} \theta_\tau = -\Sigma_\tau \nabla_\theta U_\tau (\theta_\tau) \,{\rm d}\tau + \Sigma_\tau^{1/2} {\rm d}W_\tau,
\end{equation}
where the potential $U_\tau$ now satisfies the elliptic PDE
\begin{equation*}
\nabla_\theta \cdot (\pi_\tau \Sigma_\tau \nabla_\theta U_\tau) = -(\Psi_{\rm post} - \pi_\tau \left[\Psi_{\rm post}\right])\,\pi_\tau  
\end{equation*}
instead of (\ref{eq:PDE}). Both formulations are equivalent for Gaussian PDFs $\pi_\tau$. While (\ref{eq:KVSDE}) is closer to the homotopy approach of Section \ref{sec:transform}, formulation (\ref{eq:KVSDE2}) has a closer resemblance to overdamped Langevin dynamics.
\end{remark}

\subsection{Asymptotic behaviour}

The case of a general negative log-likelihood function $\Psi_{\rm data}$ is more delicate. However, 
if the associated posterior (\ref{eq:posterior}) is non-Gaussian but is well approximated by a Gaussian,
and this also applies to all intermediate PDFs $\pi_\tau$, that is,
\begin{equation} \label{eq:GaussA}
\pi_\tau(\theta) \approx \tilde \pi_\tau(\theta) 
\propto \exp \left( -\frac{1}{2}(\theta - m_\tau)^{\rm T}\Sigma_\tau^{-1}(\theta - m_\tau) \right),
\end{equation}
then following the arguments from Section \ref{sec:second-order}, the first two moments satisfy 
\begin{equation}\label{eq:a1}
\frac{\rm d}{{\rm d}\tau} m_\tau = -\Sigma_t \tilde \pi_\tau[\nabla_\theta \Psi_{\rm post}]
\end{equation}
and
\begin{equation} \label{eq:a2}
\frac{\rm d}{{\rm d}t} \Sigma_\tau = -\Sigma_t \tilde \pi_\tau [D_\theta^2 \Psi_{\rm post} ] \Sigma_t + \Sigma_t\,.
\end{equation}
The equation for the covariance matrix is stable if $\tilde \pi_\tau [D_\theta^2 \Psi_{\rm post} ]$ is positive-definite for all
$\tau \ge 0$. This is, for example, the case for logistic regression.

We now discuss the limiting $\tau\to \infty$ behavior of (\ref{eq:KVSDE}) in some more detail. Its equilibrium distributions are characterised by the following lemma.

\begin{lemma}
Stationary measures $\pi_\infty$ of the mean-field equation (\ref{eq:KVSDE}) satisfy the PDE
$$
\frac{1}{2} \nabla_\theta \cdot \left( \pi_\infty  \Sigma_\infty \left\{ 
\nabla_\theta \log \pi_\infty   + \Sigma_{\rm prior}^{-1}(\theta +m_\infty -2m_{\rm prior})    \right\} \right) = \pi_\infty (\Psi_{\rm data}-
\pi_\infty [\Psi_{\rm data}]) 
$$
or, in terms of expectation values with respect to test functions $\phi$,
\begin{subequations} \label{eq:weak}
\begin{align}
&\frac{1}{2} \pi_\infty \left[ \nabla_\theta \phi \cdot \Sigma_\infty \left\{ \nabla_\theta \log \pi_\infty  + 
\Sigma_{\rm prior}^{-1}(\theta+m_\infty -2m_{\rm prior})    \right\} \right] \\
&\qquad \qquad \qquad  \qquad \qquad \qquad \qquad \qquad 
\qquad
= \pi_\infty \left[ \phi \left(\Psi_{\rm data}-\pi_\infty \left[ \Psi_{\rm data} \right] \right) \right]\,.
\end{align}
\end{subequations}
\end{lemma}

\begin{proof} The lemma follows from the Fokker--Planck equation associated to (\ref{eq:KVSDE}) and the PDE (\ref{eq:PDE}).
\end{proof}

\noindent
Again making the Gaussian approximation (\ref{eq:GaussA}), formally setting $\tau = \infty$ and recalling Remark \ref{rem:KVSDE}, we first obtain from (\ref{eq:weak}) the identity
$$
\frac{1}{2} \tilde \pi_\infty \left[ \nabla_\theta \phi \cdot \Sigma_\infty \nabla_\theta \log \tilde \pi_\infty  \right] 
= \tilde \pi_\infty \left[ \phi \left(\Psi_{\rm post}- \tilde \pi_\infty \left[ \Psi_{\rm post} \right] \right) \right]\,.
$$
Then,
\begin{equation}\label{eq:a1e}
\frac{1}{2} \tilde \pi_\infty [ \nabla_\theta \log \tilde \pi_\infty] = 0 = \tilde \pi_\infty [ \nabla_\theta \Psi_{\rm post}] 
\end{equation}
for $\phi(\theta) = \theta-m_\infty$, and
\begin{equation} \label{eq:a2e}
\tilde \pi_\infty [ D^2_\theta \log \tilde \pi_\infty] = \Sigma_\infty^{-1} = \tilde \pi_\infty [ D^2_\theta \Psi_{\rm post}]
\end{equation}
for $\phi(\theta) = (\theta-m_\infty)(\theta-m_\infty)^{\rm T}$, which provide a self-consistent system of equations for the mean 
$m_\infty$ and the covariance matrix $\Sigma_\infty$ of the Gaussian approximation (\ref{eq:GaussA}). Here we
have again used (\ref{eq:Opper1}) and (\ref{eq:Opper2}), respectively. 

Note that (\ref{eq:a1e}) and (\ref{eq:a2e}) are compatible with (\ref{eq:a1}) and (\ref{eq:a2}), respectively. However, (\ref{eq:a1e}) and (\ref{eq:a2e}) only require Gaussianity (or a near-Gaussianity) for the equilibrium distribution $\pi_\infty$.

\subsection{Numerical implementation}

The most appealing aspect of our modified overdamped Langevin dynamics is the possibility of implementing the SDEs (\ref{eq:KVSDE}) or
(\ref{eq:KVSDE2}), respectively, in a 
derivative-free manner, that is, without explicit knowledge of the potential $V_\tau$ and its gradient. 

We describe a suitable time-stepping method in detail for the formulation (\ref{eq:KVSDE}) utilising methods from sequential data assimilation,
which are well suited to deal with the $-\Sigma_\tau \nabla_\theta V_\tau$ 
term in the SDE (\ref{eq:KVSDE}). In other words, we numerically solve the SDE (\ref{eq:KVSDE}) with a step size $\Delta \tau$ by alternating in each time step between (i) a data assimilation step 
(such as SMC or an ensemble transform filter \cite{reichcotter15}) with negative log-likelihood 
$$
l(\theta) = \Delta \tau \Psi_{\rm data}(\theta)
$$
and (ii) the SDE
\begin{equation}
\label{eq:diffusion}
{\rm d}\theta_t = -\frac{1}{2}\Sigma_\tau
\Sigma_{\rm prior}^{-1}(\theta_\tau+m_\tau -2m_{\rm prior})  \,{\rm d}\tau + \Sigma_t^{1/2} {\rm d} W_t
\end{equation}
over the same time-interval $\Delta \tau$. More precisely, given $\theta_{\tau_k} \sim \pi_{\tau_k}$ we first find
\begin{equation} \label{eq:importance}
\tilde \theta_{\tau_{k+1}} \sim \exp(-\Delta \tau \Psi_{\rm data})\,\pi_{\tau_k}
\end{equation}
followed by solving (\ref{eq:diffusion}) with initial $\tilde \theta_{\tau_{k+1}}$ over the time-interval $\Delta \tau$ in order to obtain
$\theta_{\tau_{k+1}} \sim \pi_{\tau_{k+1}}$ for $k=0,\ldots,K-1$. Note that (\ref{eq:diffusion}) can be solved robustly by the
tamed forward Euler-Maruyama method
$$
\theta_{\tau_{k+1}} = \tilde \theta_{\tau_{k+1}} - \frac{\Delta \tau}{2} \tilde \Sigma_{\tau_{k+1}}
\left(\Sigma_{\rm prior} +  \Delta \tau \tilde \Sigma_{\tau_{k+1}} \right)^{-1} (\tilde \theta_{\tau_{k+1}}+ \tilde m_{\tau_{k+1}} 
-2m_{\rm prior}) + \sqrt{\Delta \tau} \tilde \Sigma_{\tau_{k+1}}^{1/2} \Xi_{k+1}
$$
with $\Xi_{k+1}$ a $D$-dimensional standard Gaussian random variable. As for ALDI, the $\Sigma_\tau^{1/2}$ multiplying the Brownian motion in (\ref{eq:diffusion}) requires a correction term in order to sample from the correct distribution for finite ensemble sizes $M$. See \cite{GINR19} for details. 

Our numerical implementations of (\ref{eq:importance}) in terms of ensembles $\{\theta_{\tau_k}^{(i)}\}$ 
rely on ensemble transform filters of the form \cite{reichcotter15}
\begin{equation} \label{eq:ETF}
\tilde \theta_{\tau_{k+1}}^{(j)} = \sum_{i=1}^M \theta_{\tau_k}^{(i)} s_{ij}
\end{equation}
for $j =1,\ldots,M$ with appropriate coefficients $s_{ij}$, which satisfy
\begin{equation}\label{eq:constraint1}
\sum_{i=1}^M s_{ij} = 1.
\end{equation}
We collect the coefficients $s_{ij}$ in the $M\times M$ matrix $S$.

\begin{lemma}
An ensemble transform filter is affine-invariant (compare Section  \ref{sec:affine_invariance}) if the associated coefficients $\overline{s}_{ij}$ for the filter in the transformed variable $\overline{\theta}$ satisfy
$s_{ij} = \overline{s}_{ij}$ for all $i,j = 1,\ldots,M$.  
\end{lemma}

\begin{proof}
If 
$$
\theta_{\tau_k}^{(i)} = A\overline{\theta}_{\tau_k}^{(i)}+b
$$
at $\tau_k$, then
$$
\theta_{\tau_{k+1}}^{(j)} = \sum_{i=1}^M \theta_{\tau_k}^{(i)}s_{ij} =
\sum_{i=1}^M (A \overline{\theta}_{\tau_k}^{(i)} + b)s_{ij}
= A \sum_{i=1}^M  \overline{\theta}_{\tau_k}^{(i)} \overline{s}_{ij} + b
= A \overline{\theta}_{\tau_{k+1}}^{(j)} + b\,,
$$
and we conclude that 
$$
\theta_{\tau_{k+1}}^{(i)} = A\overline{\theta}_{\tau_{k+1}}^{(i)}+b
$$
at $\tau_{k+1}$ for all $i=1,\ldots,M$.
\end{proof}

\begin{table}
    \centering
    \begin{tabular}{|c | c | c | c | c|}
    \hline & & & &\\
        ensemble size/method & $M=50$ & $M=100$ & $M=200$ & $M=400$ \\ & & & & \\ \hline  & & & & \\
         McKean--Vlasov SDE & $\left( \begin{array}{c}-3.34\\-3.39 \\ 3.22 \end{array}\right)$ & $\left( \begin{array}{c}-3.35\\-3.39 \\ 3.21 \end{array}\right)$ &
         $\left( \begin{array}{c}-3.35\\-3.39 \\ 3.21 \end{array}\right)$ &
         $\left( \begin{array}{c}-3.35\\-3.39 \\ 3.22 \end{array}\right)$ \\
         & & & & \\ \hline & & & & \\
         FPF & $\left( \begin{array}{c}-3.31\\-3.35 \\ 3.20 \end{array}\right)$ & $\left( \begin{array}{c}-3.33\\-3.37 \\ 3.21 \end{array}\right)$ &
         $\left( \begin{array}{c}-3.33\\-3.37 \\ 3.21 \end{array}\right)$ &
         $\left( \begin{array}{c}-3.33\\-3.37 \\ 3.21 \end{array}\right)$ \\
         & & & & \\ \hline & & & & \\
         second-order & $\left( \begin{array}{c}-3.36\\-3.40 \\ 3.22 \end{array}\right)$ & $\left( \begin{array}{c}-3.36\\-3.41 \\ 3.22 \end{array}\right)$ &
         $\left( \begin{array}{c}-3.36\\-3.41 \\ 3.22 \end{array}\right)$ &
         $\left( \begin{array}{c}-3.36\\-3.41 \\ 3.22 \end{array}\right)$ \\
         & & & & \\ \hline & & & & \\
         EnKBF & $\left( \begin{array}{c}-3.27\\-3.31 \\ 3.20 \end{array}\right)$ & $\left( \begin{array}{c}-3.27\\-3.32 \\ 3.20 \end{array}\right)$ &
         $\left( \begin{array}{c}-3.27\\-2.31 \\ 3.19 \end{array}\right)$ &
         $\left( \begin{array}{c}-3.27\\-3.31 \\ 3.19 \end{array}\right)$ \\ & & & & \\ 
         \hline
    \end{tabular} \medskip
    \caption{Informative prior: We display the ensemble mean averaged over $L=1000$ repeated experiments for the four proposed methods as a function of the ensemble size $M$. The reference value from an ALDI simulation is $(-3.32,-3.36,3.20)^{\rm T}$. We find that all methods reproduce the reference solution provided by ALDI within small variations.}
    \label{tab:example_2a}
\end{table}

\noindent
The EnKF and the nonlinear ensemble transform filter (NETF) of \cite{Toedter16}, for example, are affine-invariant. More specifically, the NETF leads to the transformation matrix 
$$
S = w_{\tau_k} {\rm 1}^{\rm T}_M + \sqrt{M} \left( {\rm diag}\,(w_{\tau_k}) - w_{\tau_k} w_{\tau_k}^{\rm T}\right)^{1/2}
$$
where $w_{\tau_k} \in \R^M $  is the vector of normalised importance weights with entries
$$
w_{\tau_k}^{(i)} = \frac{\exp(-\Delta \tau \Psi_{\rm data}(\theta_{\tau_k}^{(i)}))}{\sum_{j=1}^M \exp(-\Delta \tau \Psi_{\rm data}(\theta_{\tau_k}^{(j)}))},
$$
${\rm 1}_M \in \R^M$ is the vector of ones, and $\mbox{diag}\,(w_{\tau_k}) \in \R^{M\times M}$ is the diagonal matrix with diagonal entries
$w_{\tau_k}$. The affine invariance follows from the invariance of the importance weights $w_{\tau_k}$. The affine invariance of the EnKF has been previously demonstrated, for example, in \cite{NP15}.

\begin{table}
    \centering
    \begin{tabular}{| c | c | c | c | c |}
    \hline & & & &\\
        ensemble size/method & $M=50$ & $M=100$ & $M=200$ & $M=400$ \\ & & & & \\ \hline  & & & & \\
         McKean--Vlasov SDE & $0.89$ & $0.83$ & $0.80$ & $0.78$ \\
         & & & & \\ \hline & & & & \\
         FPF & $0.98$ & $0.90$ & $0.86$ & $0.83$ \\
         & & & & \\ \hline & & & & \\
         second-order & $0.75$ & $0.75$ & $0.74$ & $0.74$ \\
         & & & & \\ \hline & & & & \\
         EnKBF & $0.80$ & $0.79$ & $0.78$ & $0.78$ 
         \\ & & & & \\ 
         \hline
    \end{tabular} \medskip
    \caption{Informative prior: We display the spectral norm of the final ensemble covariance matrices averaged over $L=1000$ repeated experiments for the four proposed methods as a function of the ensemble size $M$. The ALDI reference value is $0.82$. We find that the FPF slightly overestimates the variance for smaller ensemble sizes, while both the EnKBF and the second-order formulation slightly underestimate it across all ensemble sizes.}
    \label{tab:example_2b}
\end{table}

\begin{remark}
The ensemble transform particle filter (ETPF) of \cite{reich13,reichcotter15} can also be made affine-invariant by using the cost function
\begin{equation} \label{eq:ETPF}
V(S) = \frac{1}{2}\sum_{i,j=1}^M s_{ij}\, (\theta_{\tau_k}^{(i)}-\theta_{\tau_k}^{(j)})^{\rm T} \hat \Sigma_{\tau_k}^{-1}
(\theta_{\tau_k}^{(i)}-\theta_{\tau_k}^{(j)})
\end{equation}
subject to the constraints $s_{ij}\ge 0$, (\ref{eq:constraint1}), and
\begin{equation} \label{eq:constraint2}
\frac{1}{M}\sum_{j=1}^M s_{ij} = w_{\tau_k}^{(i)}
\end{equation}
in the optimal transport definition of the coefficients $s_{ij}$ in (\ref{eq:ETF}). Again, the affine invariance follows from the affine invariance of both the cost function $V(S)$ and the importance weights $w_{\tau_k}$. 
\end{remark}

%
\section{Numerical example} \label{sec:numerics}
%

In this section, we provide a couple of relatively simple numerical illustrations for the methods proposed in this paper. We start with a low-dimensional toy problem for which all proposed methods can easily be implemented and tested.

\begin{table}
    \centering
    \begin{tabular}{|c | c | c | c | c|}
    \hline & & & &\\
        ensemble size/method & $M=50$ & $M=100$ & $M=200$ & $M=400$ \\ & & & & \\ \hline  & & & & \\
         McKean--Vlasov SDE & $\left( \begin{array}{c}-2.61\\-2.65 \\ 2.15 \end{array}\right)$ & $\left( \begin{array}{c}-2.60\\-2.64 \\ 2.15 \end{array}\right)$ &
         $\left( \begin{array}{c}-2.59\\-2.62 \\ 2.15 \end{array}\right)$ &
         $\left( \begin{array}{c}-2.59\\-2.62 \\ 2.15 \end{array}\right)$ \\
         & & & & \\ \hline & & & & \\
         FPF & $\left( \begin{array}{c}-2.16\\-2.15 \\ 1.64 \end{array}\right)$ & $\left( \begin{array}{c}-2.28\\-2.30 \\ 1.83 \end{array}\right)$ &
         $\left( \begin{array}{c}-2.39\\-2.41 \\ 1.95 \end{array}\right)$ &
         $\left( \begin{array}{c}-2.44\\-2.47 \\ 2.00 \end{array}\right)$ \\
         & & & & \\ \hline & & & & \\
         second-order & $\left( \begin{array}{c}-2.29\\-2.32 \\ 1.82 \end{array}\right)$ & $\left( \begin{array}{c}-2.30\\-2.32 \\ 1.83 \end{array}\right)$ &
         $\left( \begin{array}{c}-2.31\\-2.33 \\ 1.83 \end{array}\right)$ &
         $\left( \begin{array}{c}-2.31\\-3.34 \\ 1.84 \end{array}\right)$ \\
         & & & & \\ \hline & & & & \\
         EnKBF & $\left( \begin{array}{c}-2.14\\-2.16 \\ 1.73 \end{array}\right)$ & $\left( \begin{array}{c}-2.16\\-2.18 \\ 1.74 \end{array}\right)$ &
         $\left( \begin{array}{c}-2.17\\-2.19 \\ 1.75 \end{array}\right)$ &
         $\left( \begin{array}{c}-2.17\\-2.19 \\ 1.76 \end{array}\right)$ \\ & & & & \\ 
         \hline
    \end{tabular} \medskip
    \caption{Less informative prior: We display the ensemble mean averaged over $L=1000$ repeated experiments for the four proposed methods as a function of the ensemble size $M$. We find that all methods qualitatively reproduce the reference value $(-2.56,-2.59,2.15)^{\rm T}$ provided by ALDI.  While the second-order method and the EnKBF suffer from a systematic bias, the FPF approaches the reference solution as the ensemble size increases. The McKean--Vlasov SDE formulation yields good approximations for all ensemble sizes.}
    \label{tab:example_2c}
\end{table}

\begin{example} \label{example2}
We follow the example from \cite{Bishop06} Section 4.2.1.
Specifically, consider the Gaussian likelihood functions
$$
\pi(x|C_i) \propto \exp\left(-\frac{1}{2}
(x-\mu_i)^{\rm T} B^{-1}(x-\mu_i)\right)
$$
for inputs $x$ to belong to classes $C_i$, $i=1,2$. The implied posterior PDF $\pi(C_1|\phi_x)$ is given by (\ref{eq:conditional}) with $\phi_x = (x^{\rm T},1)^{\rm T}$, that is $D = J+1$, and the true parameter value is 
\begin{equation} \label{eq:ex2_theta_opt}
\theta = \left( \begin{array}{c}
B^{-1}(\mu_1-\mu_2)\\
-\frac{1}{2}\mu_1^{\rm T} B^{-1}\mu_1 + \frac{1}{2}
\mu_2^{\rm T}B^{-1} \mu_2 + \ln \frac{\pi(C_1)}{\pi(C_2)}
\end{array} \right)\,.
\end{equation}
Furthermore, given data points $(x_n,t_n)$, $n=1,\ldots,N$, the maximum likelihood estimators for $\pi(C_1)$, $\mu_i$, and $B$ are given by
$$
\hat \pi(C_1) = \frac{N_1}{N}, \qquad N_1 = \sum_{n=1}^N t_n,
$$
$$
\hat \mu_1 = \frac{1}{N_1}\sum_{n=1}^N t_n x_n,\qquad
\hat \mu_2 = \frac{1}{N-N_1} \sum_{n=1}^N (1-t_n)x_n\,,
$$
and
$$
\hat B = \frac{1}{N} \left\{
\sum_{n=1}^N t_n (x_n-\hat \mu_1)(x_n-\hat \mu_1)^{\rm T} +
\sum_{n=1}^N (1-t_n)(x_n-\hat \mu_2)(x_n-\hat \mu_2)^{\rm T} \right\}\,,
$$
respectively. See Section 4.2.2 in \cite{Bishop06} for details.
In this example, we will however directly estimate the true
parameter value (\ref{eq:ex2_theta_opt}) using Bayesian inference.

We consider the special case $J = 2$ and use $B = {\rm I}$, ${\rm I}$ the identity matrix, $\mu_1 =
(-1,-1)^{\rm T}$, $\mu_2 = (2,2)^{\rm T}$, and $\pi(C_1) = \pi(C_2) = 1/2$ to generate $N=100$ data points. We note that the chosen values for $\mu_1$, $\mu_2$, and $B$ lead to 
\begin{equation} \label{eq:ex2_theta}
\theta = \left( \begin{array}{c} -3 \\ -3 \\ 3
\end{array} \right).
\end{equation}

We now test the proposed algorithms using (i) an informative Gaussian prior
with mean (\ref{eq:ex2_theta}) and covariance matrix $\Sigma_{\rm prior} = {\rm I}$,
and (ii) a less informative Gaussian prior with mean $m_{\rm prior} = 0$ and covariance matrix $\Sigma_{\rm prior} = 4\,{\rm I}$. We time step the affine-invariant FPF, the second-order filter (\ref{eq:SO1a})-(\ref{eq:SO1b}), and the
EnKBF (\ref{eq:KBF1a}) with the forward Euler method using a step size $\Delta \tau = 10^{-3}$.  A further reduction of the step size did not change the results in a statistically significant manner. The McKean--Vlasov SDE
(\ref{eq:KVSDE}) is also implemented with the NETF as the inner data assimilation step. Furthermore, we implemented the ALDI method \cite{GINR19} for comparison since it is known to exactly sample from the posterior PDF (\ref{eq:posterior}). Both ALDI and the McKean--Vlasov SDE (\ref{eq:KVSDE2}) were run up to time $\tau = 10$ with step-size $\Delta \tau = 10^{-2}$ at which point the interacting particle systems were considered to be in equilibrium. 

The implementation of the FPF, as well as the SDE (\ref{eq:KVSDE}), explicitly involve the evaluation of the negative log-likelihood function (\ref{eq:nlf}). We found that we needed to replace (\ref{eq:y_x}) by
$$
y_n = 0.99 \,\sigma (\theta^{\rm T}\phi_{x_n}) + 0.005
$$
in order to avoid numerical instabilities due to exceedingly small values of $\ln y_n$. The band-width, $\epsilon$, in (\ref{eq:kernel}) has been set to $\epsilon = 0.1$.

All experiments were repeated $L = 1000$ times in order to average out random sampling effects. The numerically computed posterior means and spectral norm of the posterior covariance matrices can be found in Tables \ref{tab:example_2a} and \ref{tab:example_2b}, respectively, for the informative prior, and in Tables 
\ref{tab:example_2c} and \ref{tab:example_2d}, respectively, for the less informative prior. Reference values from an ALDI simulation are provided in the captions. While all methods reproduce those reference values very well in case of the informative prior, this picture changes for the less informative prior, where one expects a stronger nonlinear and non-Gaussian behavior of the associated Bayesian inference problem. One finds that the accuracy of the FPF improves as the ensemble size increases while the second-order method and the EnKBF suffer from a systematic bias. We also find that the McKean--Vlasov SDE formulation (\ref{eq:KVSDE}) behaves rather well over the full range of ensemble sizes. We conclude that all tested methods are able to qualitatively reproduce the exact parameter value (\ref{eq:ex2_theta}) even in the case of the less informative prior. The results also indicate that the posterior distribution is close to Gaussian, while the intermittent PDFs $\pi_\tau$ must deviate significantly from a Gaussian distribution.

In terms of computational cost, the EnKBF and the second-order method perform comparable and are the least expensive by far for all ensemble sizes. While the McKean--Vlasov SDE is twice as expensive as the EnKBF at $M = 50$, this ratio goes up to a factor of five at $M = 400$. The FPF, on the other hand, is ten times as expensive as the EnKFB at $M =50$ which increases  steeply to a factor of one hundred at $M = 400$.

\end{example}

\begin{table}
    \centering
    \begin{tabular}{| c | c | c | c | c |}
    \hline & & & &\\
        ensemble size/method & $M=50$ & $M=100$ & $M=200$ & $M=400$ \\ & & & & \\ \hline  & & & & \\
         McKean--Vlasov SDE & $1.36$ & $1.14$ &
         $1.03$ & $0.97$ \\
         & & & & \\ \hline & & & & \\
         FPF & $3.61$ & $2.58$ & $2.01$ & $1.66$ \\
         & & & & \\ \hline & & & & \\
         second-order & $0.46$ & $0.45$ & $0.45$ & $0.45$ \\
         & & & & \\ \hline & & & & \\
         EnKBF & $0.60$ & $0.59$ & $0.59$ & $0.59$ 
         \\ & & & & \\ 
         \hline
    \end{tabular} \medskip
    \caption{Less informative prior: We display the spectral norm of the final ensemble covariance matrices averaged over $L=1000$ repeated experiments for the four proposed methods as a function of the ensemble size $M$. The ALDI reference value is $1.18$.We find that the FPF overestimates the variance, while both the EnKBF and the second-order formulation underestimate it. While the FPF appears to approach the correct reference value for increased ensemble sizes, both the EnKBF and the second-order filter appear to suffer from a systematic bias. The McKean--Vlasov SDE approach leads to a slightly degraded performance in terms of ensemble variations as the ensemble size increases.}
    \label{tab:example_2d}
\end{table}

\noindent
We now consider a high-dimensional extension of the previous example, for which we further investigate the performance of 
the EnKBF (\ref{eq:KBF2a}) in terms of the tamed time-stepping method (\ref{eq:tamed_integrator}) for ensemble sizes $M\le D$, localisation via dropouts, and data sub-sampling.

\begin{example} \label{example3}
We consider the simple feature map
$$
\phi_x = x,
$$
that is $J=D$ such that 
$$
y_x(\theta) = \sigma(\theta^{\rm T}x)
$$
for $D=50$. The data is generated by first drawing a $\theta = \theta_{\rm ref}$ from the standard normal distribution in $\R^D$. Next, $N=1000$ data points are generated by randomly drawing $x_n$, $n= 1,\ldots,N$, again from 
a standard normal distribution in $\R^D$. Those points are assigned the label
$t_n = 1$ with probability $y_{x_n}(\theta_{\rm ref})$, and $t_n = 0$ otherwise. We use the EnKBF tamed time-stepping method (\ref{eq:tamed_integrator}) with step size $\Delta \tau = 1/200$ for parameter inference. The initial ensemble $\theta_0^{(i)}$, $i=1,\ldots,M$ is drawn from the prior Gaussian distribution with mean $m_{\rm prior}=0$ and the covariance matrix $\Sigma_{\rm prior}$ equal to the identity matrix. We vary the ensemble size $M$ between $M=20$ and $M = 100$ and use dropout localisation for the computation of the empirical covariance matrices $\hat \Sigma_{\tau_n}$ following (\ref{eq:localisation}). More specifically, an entry of $\theta_{\tau_n}^{(i)}$ is set to zero if
$\eta<0.5$ where the $\eta$'s are i.i.d.~ uniform random variables from the interval $[0,1]$. We report  in Table \ref{tab:example_3a}
the $l_2$-difference between $\theta_{\rm ref}$
and the ensemble mean $\hat m_\tau$ at final time $\tau = 1$. We also state the spectral norm of $\hat \Sigma_\tau$ at the final time in Table \ref{tab:example_3b}. In a second set of experiments, we apply mini-batching to the localised EnKBF with random batches of size $N' = 100$. We also compare the results to those from the standard EnKBF formulation (\ref{eq:KBF1a}).  Furthermore, ALDI is run for a sample size of $M=100$ in order to provide the numerical benchmark values listed in the captions of Tables \ref{tab:example_3a} and \ref{tab:example_3b}. All results have been averaged over $L = 1000$ independent simulations. We also report the standard deviations in addition to the averaged values. Our numerical results indicate that localisation by dropout is effective for small ensemble sizes, and that mini-batching does not degrade performance while providing significant computational savings. As for the previous example, the EnKBF is able to capture the posterior mean quite well while the posterior variances are systematically underestimated. Ensemble inflation, as widely used for the EnKF \cite{evensen,reichcotter15}, could help to further improve the variance estimates. We found that our results are insensitive to the choice of the step size $\Delta \tau$ in (\ref{eq:tamed_integrator}). However changing the threshold value $\mu = 0.5$ in the dropout localisation (\ref{eq:localisation}) affects the results significantly. Smaller values of $\mu$ lead in particular to reduced $l_2$-differences for larger ensemble sizes, $M$, while being less effective for smaller $M$. See Table \ref{tab:example_3c}.

\end{example}

\begin{table}
    \centering
     \begin{tabular}{| c | c | c | c | c | c|}
    \hline & & & & &\\
        ensemble size/method & $M=20$ & $M=40$ & $M=60$ & $M=80$ & $M=100$ \\ & & & & & \\ \hline  & & & & & \\
         EnKBF & 6.26$\pm0.75$ & 4.55$\pm$0.83 & 2.67$\pm$0.68 & 1.99$\pm$0.56 & 1.69$\pm$0.48\\ & & & & &\\ \hline & & & & &\\
         EnKBF  & 1.29$\pm$0.29 & 1.19$\pm$0.23 & 1.28$\pm$0.30 & 1.35$\pm$0.34 & 1.39$\pm$0.46\\
         (with dropout) & & & & &\\ \hline & & & & & \\
         EnKBF  & 2.14$\pm$0.46 & 1.56$\pm$0.30 & 1.42$\pm$0.25 & 1.37$\pm$0.22 & 1.35$\pm0.21$
         \\ (with dropout \& mini-batch) & & & & & \\ 
         \hline 
    \end{tabular}
     \medskip
    \caption{We display the $l_2$-difference between the posterior ensemble means and the true parameter value averaged over $L=1000$ repeated experiments  and its standard deviation for three different implementations of the EnKBF as a function of the ensemble size $M$. The reference value provided by ALDI is $1.10$. We find that localisation by dropout dramatically improves the behavior of the EnKBF especially for small ensemble sizes. Using mini-batches reduces the approximation quality slightly while providing significant computational savings.}
    \label{tab:example_3a}
\end{table}

%
\section{Generalisations and extensions} \label{sec:extensions}
%

In this section, we discuss possible extensions of the proposed methods to nonlinear logistic regression including derivative-free implementations, multi-class logistic regression, and sigmoidal Cox processes.

\subsection{Nonlinear logistic regression}

We generalise linear logistic regression from Section \ref{sec:background} to the more general logistic regression problem with 
$y_x(\theta) := \sigma(f_x(\theta))$ for given functions 
$f_x:\R^D \to \R$. Hence the gradient $\nabla_\theta y_x(\theta)$ becomes
$$
\nabla_\theta y_x(\theta) = \phi_x(\theta) \,(1-y_x(\theta))\,y_x(\theta),
$$
where the previously considered feature maps $\phi_x$ now also depend on the parameters $\theta$, that is,
$$
\phi_x(\theta) := \nabla_\theta f_x(\theta).
$$
The gradient of $\Psi_{\rm data}$ is thus explicitly given by
$$
\nabla_\theta \Psi_{\rm data}(\theta) = \sum_{n=1}^N  \phi_{x_n}(\theta)\,(y_{x_n}(\theta)-t_n)
$$
and the Hessian matrix of second-order derivatives becomes
$$
D^2_\theta \Psi_{\rm data}(\theta) = 
\sum_{n=1}^N \left\{ \phi_{x_n}\,(\theta) \,y_{x_n}(\theta)\,(1-y_{x_n}(\theta))\,\phi_{x_n}(\theta)^{\rm T} +
D_\theta \phi_{x_n}(\theta)\,(y_{x_n}(\theta)-t_n) \right\}.
$$

\medskip

\noindent
We now discuss  a derivative-free implementation of the EnKBF (\ref{eq:KBF1a}) for nonlinear logistic regression. 
The key difference to linear logistic regression is that the matrix $\Phi \in \R^{D\times N}$ now depends on $\theta$ and is given by
$$
\Phi(\theta) = (\nabla_\theta f_{x_1}(\theta),\ldots,\nabla_\theta f_{x_N}(\theta)).
$$
In order to avoid the computation of the gradients of $f_x$, one replaces $\Sigma_\tau \Phi(\theta_\tau)$ in (\ref{eq:KBF1a}) 
by the empirical correlation matrix
$$
C_\tau := \frac{1}{M-1} \sum_{i=1}^M(\theta_\tau^{(i)}-\hat{m}_\tau)\left( f_{x_1}(\theta_\tau^{(i)}), \ldots,
f_{x_N}(\theta_\tau^{(i)})\right) \in \R^{D\times N}
$$
between $\theta_\tau$ and $f_{x_n}(\theta_\tau)$ for $n=1,\ldots,N$. See  \cite{evensen} for the basic idea in
the context of the EnKF and also \cite{KS19,HLR18} for the specific application of this idea in the context of 
machine learning. See also the recent work \cite{PSV21,RW21} on tuneable derivative-free approximations. 
We also note that this substitution becomes exact in the case of linear functions $f_x(\theta)$. Such a linear approximation
is, for example, justified for functions $f_x(\theta)$ defined by sufficiently wide neural networks \cite{Lee_2020}.

\begin{table}
    \centering
    \begin{tabular}{| c | c | c | c | c | c|}
    \hline & & & & &\\
        ensemble size/method & $M=20$ & $M=40$ & $M=60$ & $M=80$ & $M=100$ \\ & & & & & \\ \hline  & & & & & \\
         EnKBF & 0.014$\pm$0.004 & 0.034$\pm$0.009 & 0.058$\pm$0.012 & 0.079$\pm$0.014 & 0.097$\pm$0.015\\ & & & & &\\ \hline & & & & &\\
         EnKBF & 0.043$\pm$0.012 & 0.071$\pm$0.019 & 0.088$\pm$0.022 & 0.100$\pm$0.024 & 0.109$\pm$0.025 \\
         (with dropout) & & & & &\\ \hline & & & & & \\
         EnKBF  & 0.041$\pm$0.012 & 0.066$\pm$0.018 & 0.084$\pm$0.021 & 0.096$\pm$0.024 & 0.105$\pm$0.025
         \\ (with dropout and mini-batch) & & & & & \\ 
         \hline
    \end{tabular} 
     \medskip
    \caption{We provide the spectral norm of the computed posterior covariance matrices in the setting described in Table \ref{tab:example_3a}. The reference value provided by ALDI is $0.22$. We find that both all three EnKBF implementations underestimate the variance. However, dropout localisation increases the ensembles spread which in turn improves the estimated ensemble means as displayed in Table \ref{tab:example_3a}.}
    \label{tab:example_3b}
\end{table}

\subsection{Multi-class logistic regression}

The extension from two-class to multi-call classification is straightforward, see \cite{Bishop06}. 
The posterior probabilities for $L$ classes $C_l$, $l=1,\ldots,L$, are given by
$$
\pi(C_l|\phi_x) =  \frac{\exp \left(a_{x,l}\right)}{\sum_{j=1}^L \exp \left(a_{x,j}\right)},
$$
where the activation is given by
$$
a_{x,l} := \theta_l^{\rm T}\phi_x .
$$
The adjustable parameters are now $\theta = (\theta_1^{\rm T},\ldots,\theta_L^{\rm T})^{\rm T} \in \R^{L D}$. The negative log-likelihood
function becomes
$$
\Psi_{\rm data}(\theta) = - \sum_{n=1}^N \sum_{l=1}^L  t_{nl} \ln y_{nl}(\theta)
$$
with 
$$
y_{nl}(\theta) := \frac{\exp \left(a_{x_n,l} \right)}{\sum_{j=1}^L \exp \left(a_{x_n,j}\right)}
$$
and data $(x_n,\{t_{nl}\}_{l=1}^L)$, where $t_{nl} \in \{0,1\}$ subject to $\sum_{l=1}^L t_{nl} = 1$ for all $n=1,\ldots,N$. The gradients satisfy
$$
\nabla_{\theta_l} \Psi_{\rm data}(\theta) = \sum_{n=1}^N (y_{nl}-t_{nl}) \phi_n ,
$$
$l=1,\ldots,L$. The Hessian matrix of second-order derivatives got $D\times D$ block entries
$$
D_{\theta_l} D_{\theta_j} \Psi_{\rm data}(\theta) = \sum_{n=1}^N y_{nl}(\delta_{lj} - y_{nj}) \phi_n \phi_n^{\rm T}
$$
with $\delta_{lj} = 1$ if $l=j$ and $\delta_{lj} = 0$ otherwise. 
We conclude that these formulas closely resemble the corresponding expressions for binary classification and, hence, all previously 
considered algorithms easily generalise to multi-class regression.

\begin{table}
    \centering
 \begin{tabular}{| c | c | c | c | c | c|}
    \hline & & & & &\\
        ensemble size/method & $M=20$ & $M=40$ & $M=60$ & $M=80$ & $M=100$ \\ & & & & & \\ \hline  & & & & & \\
         EnKBF  & 3.30$\pm$1.01 & 1.73$\pm$0.60 & 1.26$\pm$0.23 & 1.14$\pm$0.31 & 1.12$\pm$0.36\\
         (with dropout) & & & & &\\ \hline & & & & & \\
         EnKBF  & 3.39$\pm$0.86 & 1.86$\pm$0.61 & 1.38$\pm$0.25 & 1.24$\pm$0.23 & 1.19$\pm$0.27
         \\ (with dropout and mini-batch) & & & & & \\ 
         \hline 
    \end{tabular} 
 \medskip
    \caption{Same results as displayed in Table \ref{tab:example_3a} except that the dropout rate has been reduced from $\mu=0.5$ to $\mu = 0.2$. While the estimation errors increase for the smallest ensemble sizes the reduced rate improves the performance for larger ensemble sizes.}
    \label{tab:example_3c}
\end{table}

\subsection{Sigmoidal Cox processes} \label{sec:SCP}

We consider Cox processes (CPs) on a domain $\mathcal{X} \subset \R^J$. Following \cite{AMMK09}, 
the CP is parametrised by an intensity function $\lambda_x:\mathcal{X}\to \R^+$ which we represent using a sigmoidal 
random feature model, that is,
$$
\lambda_x(\theta) = \lambda^\ast \sigma(\theta^{\rm T} \phi_x),
$$
where $\lambda^\ast>0$ provides the upper threshold value for the intensity function. Note that a Gaussian process is used
by \cite{AMMK09} while we rely here on the closely related random feature maps $\phi_x$
\cite{LH21}, which easier integrate into the homotopy approach and the EnKF in particular
\cite{GR20}.

See \cite{DoOp18b} for an efficient Bayesian inference for CPs using variable augmentation. Here we suggest an alternative approach using an associated unbiased estimator of the negative log-likelihood in the forward Euler time-stepping (\ref{eq:FEM}).

More specifically, given $N$ events $x_n$, $n=1,\ldots,N$,
the negative log-likelihood function is provided by
$$
\Psi_{\rm data}(\theta) = \int_\mathcal{X} \lambda_x (\theta) {\rm d}x - \sum_{n=1}^N \ln \lambda_{x_n}(\theta)\,.
$$
Following the notation of Section \ref{sec:random_sub_sample} and using $I$ uniformly distributed samples $\hat x_i$, $i=1,\ldots,I$, on $\mathcal{X}$, we obtain the random unbiased estimator
\begin{equation} \label{eq:scp_ub}
\Psi_{\rm data}^\gamma (\theta) = \frac{\lambda^\ast}{I} \sum_{i=1}^I y_{\hat x_i}(\theta) - \sum_{n=1}^N \ln y_{x_n}(\theta) - N \ln \lambda^\ast 
\end{equation}
for this intractable likelihood $\Psi_{\rm data}$. Here we also used the previously introduced notation (\ref{eq:y_x}), that is, $\lambda_x(\theta) = \lambda^\ast y_x(\theta)$.

We now apply the homotopy approach (\ref{eq:TODE}) to the associated Bayesian inference problem. Following the discussion in Section \ref{sec:random_sub_sample}, the evolution converges to the correct posterior if in each time step (\ref{eq:FEM}) one uses a different realisation $\gamma_{\tau_k}$ of the unbiased estimator (\ref{eq:scp_ub}) instead of $\Psi_{\rm data}$ in (\ref{eq:PDE}) and takes the limit $\Delta \tau\to 0$. 

\begin{example} Let us briefly discuss a related problem in a strictly Gaussian setting. We start from the intractable negative log-likelihood function
$$
\Psi_{\rm data}(\theta) := \frac{1}{2}\int_0^1
\left( \theta^{\rm T} \phi_x -t\right)^2 {\rm d}x.
$$
Its gradient is provided by
$$
\nabla_\theta \Psi_{\rm data}(\theta) = 
\int_0^1 \phi_x \left( \theta^{\rm T} \phi_x -t\right)
{\rm d}x
$$
and the associated mean-field EnKBF becomes
$$
\frac{\rm d}{{\rm d}\tau} \theta_t
= -\frac{1}{2} \Sigma_\tau \int_0^1 \phi_x \left(
\left( \theta_\tau + m_\tau \right)^{\rm T} \phi_x -2t \right)
{\rm d}x\,.
$$
So far, everything is exact and no approximation has been made. We now replace the integral by its Monte Carlo approximation using $I$ randomly sampled $\hat{x}_{i,k}\sim {\rm U}[0,1]$, $i=1,\ldots,I$, and introduce a time step $\Delta \tau$ leading to
$$
\theta_{\tau_{k+1}} =
\theta_{\tau_k} - \frac{\Delta \tau}{2I}\Sigma_{\tau_k}
\sum_{i=1}^I \phi_{\hat x_{i,k}} \left(
\left(\theta_{\tau_k}+m_{\tau_k}\right)^{\rm T}
\phi_{\hat x_{i,k}} - 2t \right),
$$
where we re-sample the $\hat{x}_{i,k}$'s in each time-step
$\tau_k$, $k=0,\ldots,K-1$. A theoretical investigation now boils down to studying the limit $\Delta \tau \to 0$, which can be performed along the modified equation analysis of \cite{LTE19}. We note that the approximation error can be reduced by either increasing $I$ or decreasing the step size $\Delta \tau$.
\end{example}

%
\section{Conclusions} \label{sec:conclusions}
%

We have presented affine-invariant extensions of the popular EnKF and related methods such as the FPF to logistic regression. In addition, we have proposed a novel SDE-based affine-invariant and derivative-free sampling method which utilises the robust ensemble transform approach for Bayesian inference as an inner time-stepping method. We have also introduced a number of algorithmic choices such as localisation via dropout, linearly implicit time-stepping methods, and mini-batching the data, which help to improve the efficiency of the proposed methods. While we have only considered two simple numerical examples, future work will consider more challenging applications such as the training of the {\it Sentence Gestalt} model of natural language comprehension \cite{Rabovsky3}.  A numerical exploration of the affine-invariant stochastic homotopy formulation (\ref{eq:ALDI}) would also be of  practical and theoretical interest.

A number of theoretical questions also remain to be answered such as the convergence of the ensemble methods to their mean-field limits, the impact of the various approximations on the computed posterior representations, and the ergodicity and convergence to equilibrium properties of (\ref{eq:KVSDE}) and (\ref{eq:KVSDE2}), respectively. Finally, an embedding of the ensemble transform Langevin dynamics into a Markov chain Monte Carlo framework would allow for an exact sampling from the posterior.

\subsection*{Acknowledgements} 
SR is supported by Deutsche Forschungsgemeinschaft (DFG) - Project-ID 318763901 - SFB1294.

%
%
\bibliographystyle{plainurl}
\bibliography{bib_paper}
%
%

\section*{Appendix}

We provide an analysis of the SDE
\begin{equation}\label{eq:LMSDE}
{\rm d}\theta_\tau = -\Sigma_\tau \left\{\nabla_\theta
V_\tau(\theta_\tau)\,{\rm d}\tau - (\Delta \tau \Omega_\tau)^{1/2} {\rm d}W_\tau \right\},
\end{equation}
which is obtained from the modified SDE (\ref{eq:MSDE}) by setting $\Omega_\tau(\theta) = \Omega_\tau$, that is, by ignoring the state dependence of the diffusion term, and using the EnKBF drift term from 
linear regression \cite{br10a}, that is,
\[
\nabla_\theta V_\tau (\theta) = 
\frac{1}{2} G^{\rm T}\Gamma^{-1}(G\theta + G m_\tau - 2t).
\]
Conditioned on $W_\tau$, the evolution equation for the (conditional) mean becomes
\begin{equation}\label{eq:MSDE_cmean}
{\rm d}m_\tau = -\Sigma_\tau \left\{
G^{\rm T}\Gamma^{-1}(Gm_\tau - t){\rm d}\tau 
 - (\Delta \tau \Omega_\tau )^{1/2} {\rm d}W_\tau \right\}
\end{equation}
and that for the (conditional) covariance matrix 
\begin{equation}\label{eq:MSDE_cvar}
\frac{{\rm d}}{{\rm d}\tau}\Sigma_\tau = 
- \Sigma_\tau G^{\rm T}\Gamma^{-1}G \Sigma_\tau\,.
\end{equation}
We can further decompose (\ref{eq:MSDE_cmean}) into an equation for the mean, which we denote by $\mu_\tau$, and
the covariance $P_\tau$. We obtain the two evolution equations
\begin{equation}\label{eq:MSDE_mean}
\frac{\rm d}{{\rm d}\tau} \mu_\tau = 
-\Sigma_\tau
G^{\rm T}\Gamma^{-1}(G\mu_\tau - t)
\end{equation}
and
\begin{equation}\label{eq:MSDE_var}
\frac{{\rm d}}{{\rm d}\tau}P_\tau = 
- \Sigma_\tau G^{\rm T}\Gamma^{-1}G P_\tau
- P_\tau G^{\rm T} \Gamma^{-1} G \Sigma_\tau 
+ \Delta \tau \Sigma_\tau \Omega_\tau \Sigma_\tau ,
\end{equation}
respectively. The initial conditions are $\mu_0 = m_0$
and $P_0 = 0$.

We note that (\ref{eq:MSDE_mean}) and (\ref{eq:MSDE_cvar}) are equivalent to the equations arising from the standard Kalman--Bucy filter with solutions
$$
\mu_\tau = \mu_0 - K_\tau (G\mu_0-t)
$$
and
$$
\Sigma_\tau = \Sigma_0 - K_\tau G\Sigma_0,
$$
respectively, where $K_\tau$ denotes the Kalman gain matrix
$$
K_\tau = \tau \Sigma_0G^{\rm T}(\Gamma + \tau G \Sigma_0 G^{\rm T})^{-1}.
$$

If we now consider the fully observed case, that is $G = I$, under small measurement error covariance $\Gamma = \epsilon I$, $\epsilon \ll 1$, then we find that
$$
\mu_1 \approx \theta_{\rm MLE} = t, \qquad \Sigma_1 \approx
\Gamma = \epsilon I, \qquad P_1 < \frac{\Delta\tau}{2}  \Sigma_1 \Omega_1 \Sigma_1 \approx \epsilon^2 \frac{\Delta \tau}{2} \Omega_1.
$$
If we define 
\[
\Omega_\tau = \pi_\tau [\Omega_\tau(\theta)]
\]
for simplicity, where $\pi_\tau$ is the exact filtering distribution for the Kalman--Bucy filter, then $\Omega_0 = \mathcal{O}(\epsilon^{-2})$ while $\Omega_1 = \mathcal{O}(\epsilon^{-1})$. Hence, the variance, $P_\tau$, of the posterior conditional mean, $m_\tau$, is much smaller than the posterior (as well as the frequentist) variance, $\Sigma_\tau$, at $\tau = 1$ provided the step-size $\Delta \tau$ in (\ref{eq:LMSDE}) is chosen small enough.

\end{document}